\documentclass[reqno]{amsart}
\numberwithin{equation}{section}

\usepackage[colorlinks,linkcolor={blue}]{hyperref}

\usepackage[alphabetic,nobysame]{amsrefs}
\usepackage{bm}
\usepackage{amssymb}
\usepackage{graphicx} 
\usepackage{amscd}
\usepackage{enumerate}
\usepackage{cancel}
\usepackage[font=footnotesize]{caption}
\usepackage{dsfont}
\usepackage[colorinlistoftodos]{todonotes}
\usepackage{tikz-cd}

\definecolor{darkred}{rgb}{1,0,0} 
\definecolor{darkgreen}{rgb}{0,0.8,0}
\definecolor{darkblue}{rgb}{0,0,1}

\hypersetup{colorlinks,
linkcolor=darkblue,
filecolor=darkgreen,
urlcolor=darkred,
citecolor=darkgreen}

\makeatletter
\providecommand\@dotsep{5}
\makeatother

\definecolor{orange}{RGB}{253,85,0}
\definecolor{darkgreen}{RGB}{0,95,10}

 \newcommand{\inj}{\mathrm{inj}}

 \newcommand{\Z}{\mathds{Z}}
 \newcommand{\R}{\mathds{R}}
 \newcommand{\U}{\mathds{U}}

 \newcommand{\Wu}{W^{u}}
 \newcommand{\Ws}{W^{s}}
 
 \newcommand{\Es}{E^{s}}

 \newcommand{\FF}{\mathcal{F}}
 \newcommand{\UU}{\mathcal{U}}
 
 \newcommand{\WW}{\mathcal{W}}
 \newcommand{\KK}{\mathcal{K}}
 \newcommand{\GG}{\mathcal{G}}

 \newcommand{\CC}{\mathcal{C}}

 \newcommand{\Diff}{\mathrm{Diff}}
 \newcommand{\area}{\mathrm{area}}
 \newcommand{\id}{\mathrm{id}}
 \newcommand{\Emb}{\mathrm{Emb}}
 \newcommand{\trap}{\mathrm{trap}}

 \DeclareMathOperator{\interior}{int}

 \theoremstyle{plain}
 \newtheorem{MainThm}{Theorem}
 \newtheorem{MainCor}[MainThm]{Corollary}
 
 \newtheorem{Thm}{Theorem}[section]
 
 \newtheorem{Lemma}[Thm]{Lemma}
 \newtheorem{Cor}[Thm]{Corollary}
 
 \theoremstyle{definition}
 \newtheorem{Assumption}[Thm]{Assumption}
 \newtheorem{Definition}[Thm]{Definition}
 \newtheorem{Remark}[Thm]{Remark}
 \newtheorem{Example}[Thm]{Example}

\title[Surfaces of section for geodesic flows of closed surfaces]{Surfaces of section for geodesic flows\\ of closed surfaces} 

\author[G. Contreras]{Gonzalo Contreras}
\address{Gonzalo Contreras\newline\indent 
Centro de Investigaci\'on en Matem\'aticas\newline\indent 
A.P. 402, 36.000, Guanajuato, GTO, Mexico}
\email{gonzalo@cimat.mx}

\author[G. Knieper]{Gerhard Knieper}
\address{Gerhard Knieper\newline\indent Ruhr Universit\"at Bochum, Fakult\"at f\"ur Mathematik\newline\indent Geb\"aude IB 3/183, D-44780 Bochum, Germany}
\email{gerhard.knieper@rub.de}

\author[M. Mazzucchelli]{\\ Marco Mazzucchelli}
\address{Marco Mazzucchelli\newline\indent CNRS, UMPA, \'Ecole Normale Sup\'erieure de Lyon\newline\indent 46 all\'ee d'Italie, 69364 Lyon, France}
\email{marco.mazzucchelli@ens-lyon.fr}

\author[B.H. Schulz]{Benjamin H. Schulz}
\address{Benjamin H. Schulz\newline\indent Ruhr Universit\"at Bochum, Fakult\"at f\"ur Mathematik\newline\indent Geb\"aude IB 3/53, D-44780 Bochum, Germany}
\email{benjamin.schulz-c95@rub.de}

\thanks{Gonzalo Contreras is partially supported by CONACYT, Mexico, grant A1-S-10145. Gerhard Knieper, Marco Mazzucchelli, and Benjamin H. Schulz are partially supported by the SFB/TRR 191 ``Symplectic Structures in Geometry, Algebra and Dynamics'', funded by the Deutsche Forschungsgemeinschaft. Marco Mazzucchelli is also partially supported by the ANR grants CoSyDy (ANRCE40-0014) and COSY
(ANR-21-CE40-0002), and by the IEA project IEA00549 from CNRS}

\date{April 25, 2022}



\begin{document}

\begin{abstract}
We prove several results concerning the existence of surfaces of section for the geodesic flows of closed orientable Riemannian surfaces. The surfaces of section $\Sigma$ that we construct are either Birkhoff sections, meaning that they intersect every sufficiently long orbit segment of the geodesic flow, or at least they have some hyperbolic components in $\partial\Sigma$ as limit sets of the orbits of the geodesic flow that do not return to $\Sigma$. In order to prove these theorems, we provide a study of configurations of simple closed geodesics of closed orientable Riemannian surfaces, which may have independent interest. Our arguments are based on the curve shortening flow. 

\tableofcontents
\end{abstract}

\maketitle

\vspace{-40pt}

\section{Introduction}
\label{s:intro}

Let $N$ be a closed connected 3-manifold, and $X$ a nowhere vanishing vector field on $N$ with flow $\phi_t:N\to N$. A \emph{surface of section} for $X$ is a (not necessarily connected) immersed compact surface $\Sigma\looparrowright N$ whose boundary $\partial\Sigma$ is tangent to $X$, and whose interior $\interior(\Sigma)$ is embedded in $N\setminus\partial\Sigma$ and transverse to $X$. A surface of section allows us to read part of the dynamics of $X$ as discrete dynamics of the first return map to $\Sigma$. For this purpose, we consider the first return time function
\begin{align*}
\tau:\interior(\Sigma)\to(0,+\infty],
\qquad
\tau(z):=\inf\big\{ t>0\  \big|\ \phi_{t}(z)\in\Sigma \big\}.
\end{align*}
Here, as usual, we set $\inf\varnothing=+\infty$. The function $\tau$ is continuous on the whole $\interior(\Sigma)$, and smooth on the open subsets where it is finite. The surface of section $\Sigma$ is called a \emph{Birkhoff section} when there exists a finite constant $\ell>0$ such that, for each $z\in N$, the orbit segment $\phi_{[0,\ell]}(z)$ intersects $\Sigma$. Under this condition, outside $\partial\Sigma$ the dynamical system defined by $X$ is the suspension of the discrete dynamical system defined by the first return map 
\begin{align*}
 \psi:\interior(\Sigma)\to\interior(\Sigma),\qquad \psi(z)=\phi_{\tau(z)}(z).
\end{align*}
Notice that a Birkhoff section $\Sigma$ is not necessarily connected, but any connected component of $\Sigma$ is a Birkhoff section as well.
By means of a Birkhoff section, statements concerning 2-dimensional discrete dynamics can be translated into corresponding statements for the dynamics of vector fields in dimension 3.

The quest for Birkhoff sections has been a problem of major interest in dynamical systems since the seminal work of Poincar\'e \cite{Poincare:1912ua} and Birkhoff \cite{Birkhoff:1917vm}. In the context of Reeb vector fields on closed 3-manifolds, by a spectacular application of holomorphic curves techniques, Hofer--Wysocki--Zehnder \cite{Hofer:1998vy} established the existence of Birkhoff sections for all positively curved $3$-spheres in the $4$-dimensional symplectic vector space, equipped with the canonical contact form. A result of the first and third author \cite{Contreras:2021vx} proved the existence of Birkhoff sections for all closed contact 3-manifolds satisfying the Kupka-Smale condition: non-degeneracy of the closed Reeb orbits, and transversality of the stable and unstable manifolds of the hyperbolic closed Reeb orbits. In particular, the existence of Birkhoff sections holds for the Reeb vector field of a $C^\infty$ generic contact form on any closed 3-manifold, and for the geodesic vector field of a $C^\infty$ generic Riemannian metric on any closed surface.

An independent work of Colin--Dehornoy--Hryniewicz--Rechtman \cite{Colin:2022tq} established the existence of Birkhoff sections for those Reeb vector fields on any closed 3-manifold whose closed orbits are non-degenerate and equidistributed (i.e.~the contact volume form can be approximated, as a measure, by closed  orbits). It is not known whether the equidistribution of the closed orbits holds for the geodesic vector field of a $C^\infty$ generic Riemannian metric on a closed surface; indeed, even the fact that, for such a $C^\infty$ generic Riemannian metric, the lifts of the closed geodesics form a dense subset of the unit tangent bundle is an open problem. Nevertheless, Irie equidistribution theorem \cite{Irie:2021tf} implies that the equidistribution of the closed orbits holds for the Reeb vector field of a $C^\infty$ generic contact form on any closed 3-manifold, and together with \cite{Colin:2022tq} this provides an alternative proof of the $C^\infty$ generic existence of Birkhoff sections for closed contact 3-manifolds. We stress that, on a closed 3-manifold diffeomorphic to the unit tangent bundle of a closed surface, the Reeb flow of a $C^\infty$ generic contact form is not necessarily conjugate (or even orbitally equivalent) to the geodesic flow of a Riemannian metric.

In this paper, we focus on geodesic flows of closed orientable Riemannian surfaces $(M,g)$; hereafter, the Riemannian metrics  are assumed to be $C^\infty$. The unit tangent bundle $SM=\big\{v\in TM\ \big|\ \|v\|_g=1\big\}$ is equipped with the Liouville contact form 
$\lambda_{v}=g(v,d\pi(v)\,\cdot\,)$, where $\pi:SM\to M$ is the base projection. The geodesic vector field $X$ on $SM$ is the associated Reeb vector field, meaning that $\lambda(X)\equiv1$ and $d\lambda(X,\cdot)\equiv0$. The geodesic flow $\phi_t:SM\to SM$ is the associated Reeb flow. Its orbits have the form $\phi_t(\dot\gamma(0))=\dot\gamma(t)$, where $\gamma:\R\looparrowright M$ is a geodesic parametrized with unit speed $\|\dot\gamma\|_g\equiv1$. The closed geodesics of $(M,g)$ are precisely the base projections of the closed orbits of the geodesic flow. A closed geodesic is called simple when it is an embedded circle in $M$.

In this setting, surfaces of section had important dynamical applications, among which we mention the following ones: they played a crucial role in Bangert-Franks-Hingston's proof \cite{Bangert:1993ue,Franks:1992wu,Hingston:1993wb} of the existence of infinitely many closed geodesics on every Riemannian 2-sphere; they were employed by Contreras--Oliveira \cite{Contreras:2004vh} in order to establish the existence of an elliptic closed geodesic on a $C^2$-generic Riemannian 2-sphere, by Contreras--Mazzucchelli \cite{Contreras:2021tg} in order to prove the $C^2$-structural stability conjecture for Riemannian geodesic flows of closed surfaces, and by Knieper-Schulz \cite{Knieper:2022tp} in order to characterize Anosov Riemannian geodesic flows of closed surfaces as the $C^2$-stably transitive ones.

While the above mentioned result in \cite{Contreras:2021vx} implies the existence of a Birkhoff section for ``most'' geodesic flows of closed Riemannian surfaces, it does not provide any information concerning the topology of such a Birkhoff section. In contrast, celebrated work of Birkhoff \cite{Birkhoff:1917vm} provides explicit Birkhoff sections for two classes of geodesic flows: every positively curved Riemannian 2-sphere admits an embedded Birkhoff section diffeomorphic to an annulus, and any negatively curved Riemannian surface admits a Birkhoff section of genus one (see also \cite[Sect.~3]{Fried:1983uj}). This paper is largely inspired by the work of Birkhoff.

Our first main result provides specific Birkhoff sections for geodesic flows of closed surfaces  that do not admit contractible simple closed geodesics without conjugate points. We refer the reader to Section~\ref{ss:conjugate_points} for the background on the classical notion of  conjugate points. A simple closed geodesic $\gamma$ is called a \emph{waist}\footnote{In the literature, sometimes ``waist'' refers to a closed geodesic that is a local length minimizer in the free loop space, but is not necessary a simple closed geodesic.} when it is a local minimizer of the length functional over the free loop space. We recall that a non-degenerate simple closed geodesic is a waist if and only if it does not have conjugate points.

\begin{MainThm}
\label{mt:Birkhoff}
Let $(M,g)$ be a closed connected orientable Riemannian surface of genus $G\geq1$ that does not have any contractible simple closed geodesics without conjugate points. Its  geodesic vector field admits a Birkhoff section $\Sigma\looparrowright SM$, where $\Sigma$ is a compact connected surface of genus one and $8G-4$ boundary components that cover waists.
\end{MainThm}

Theorem~\ref{mt:Birkhoff} applies in particular to Riemannian surfaces without contractible simple closed geodesics, and we obtain the following corollary.

\begin{MainCor}\label{c:curvature_bound}
The geodesic vector field of any closed connected orientable Riemannian surface $(M,g)$ of genus $G\geq1$ admits a Birkhoff section of genus one and $8G-4$ boundary components that cover waists, provided any of the following two conditions is satisfied:
\begin{itemize}
 \item The Riemannian surface $(M,g)$ has no conjugate points.\vspace{2pt}
 
 \item The Gaussian curvature is bounded from above as
\begin{align*}
 \max (K_g) \leq \frac{2\pi}{\area(M,g)}.
\end{align*}
\end{itemize}
\end{MainCor}

\begin{proof}
If $(M,g)$ has no conjugate points, it does not have contractible closed geodesics (see, e.g., \cite[Th.~3.28]{Guillarmou:2024aa}), and Theorem~\ref{mt:Birkhoff} applies. If instead $(M,g)$ contains a contractible simple closed geodesic $\gamma$, we denote by $B\subsetneq M$ the open disk such that $\partial B=\gamma$, and Gauss--Bonnet theorem implies
\[
2\pi= \int_{B} K_g\,dm_{g}\leq \max (K_g) \area(B,g),
\]
where $m_g$ denotes the Riemannian measure.
In particular, $K_g$ attains positive values on $B$. This proves that, under the curvature bound $\max (K_g) \leq 2\pi/\area(M,g)$, there are no contractible simple closed geodesic, and Theorem~\ref{mt:Birkhoff} provides a Birkhoff section as claimed.
\end{proof}

Any oriented simple closed geodesic $\gamma$ on an oriented Riemannian surface $(M,g)$ defines two surfaces of section $A(\dot\gamma),A(-\dot\gamma)\subset SM$ diffeomorphic to annuli, which are referred to as the \emph{Birkhoff annuli} of $\gamma$ (see Section~\ref{ss:Fried}). Theorem~\ref{mt:Birkhoff} was inspired by the following statement for Riemannian 2-spheres due to Bangert \cite[Th.~2]{Bangert:1993ue}, which in turn generalizes the above mentioned result of Birkhoff for positively curved Riemannian 2-spheres. Our methods will provide a slightly simpler proof.

\begin{MainThm}[Bangert]
\label{mt:Bangert}
Let $(S^2,g)$ be a Riemannian 2-sphere, and $\gamma$ a simple closed geodesic with conjugate points whose complement $S^2\setminus\gamma$ does not contain simple closed geodesics without conjugate points. Then both Birkhoff annuli of $\gamma$ are Birkhoff sections. 
\end{MainThm}

\begin{Remark}\label{r:S2}
Every Riemannian 2-sphere has at least two simple closed geodesics with conjugate points. This follows from an addendum to the celebrated theorem of Lusternik and Schnirelmann \cite{Lusternik:1929wa}, see \cite[Th.~1.3(iii) and Prop.~4.2]{De-Philippis:2022wz}).\hfill\qed
\end{Remark}

Our third theorem is an improvement of the above mentioned result \cite{Contreras:2021vx} for the case of geodesic flows. If $K$ is a hyperbolic invariant subset for the geodesic flow, we denote as usual by $\Ws(K)$ and $\Wu(K)$ its stable and unstable manifolds.

\begin{MainThm}
\label{mt:Kupka_Smale}
Let $(M,g)$ be a closed connected orientable Riemannian surface satisfying the following two conditions:
\begin{itemize}

\item[(i)] All contractible simple closed geodesics without conjugate points are non-degenerate.

\item[(ii)] Any pair of not necessarily distinct contractible  waists $\gamma_1,\gamma_2$ $($if it exists$)$ satisfies the transversality condition $\Wu(\dot\gamma_1)\pitchfork\Ws(\dot\gamma_2)$.

\end{itemize}
Then the geodesic vector field of $(M,g)$ admits a Birkhoff section.
\end{MainThm}

The Birkhoff section provided by Theorem~\ref{mt:Kupka_Smale} is constructed by applying a surgery procedure due to Fried (see Section~\ref{ss:Fried}) to the Birkhoff annuli of a suitable finite collection of closed geodesics. The topology of such a Birkhoff section depends on the configuration of the closed geodesics in the collection, which in turn depends on the geometry of the Riemannian surface.

When an orientable Riemannian surface does not satisfy the assumptions of Theorems~\ref{mt:Birkhoff}, \ref{mt:Bangert}, and \ref{mt:Kupka_Smale}, but still satisfies a mild non-degeneracy condition, we are at least able to construct a surface of section $\Sigma$ whose escaping orbits are asymptotic to certain hyperbolic components of $\partial\Sigma$. 

\begin{MainThm}
\label{mt:broken_book}
Let $(M,g)$ be a closed connected orientable Riemannian surface  all of whose contractible simple closed geodesics without conjugate points are non-degenerate. Its geodesic vector field admits a surface of section $\Sigma\looparrowright SM$ satisfying the following properties:
\begin{itemize}

\item[(i)] \textnormal{\textbf{Topology:}} If $M=S^2$, $\Sigma$ is the disjoint union of the Birkhoff annuli of some simple closed geodesics. If $M$ has genus $G\geq 1$, $\Sigma$ is the disjoint union of the Birkhoff annuli of some simple closed geodesics and of a compact connected surface of genus one and $8G-4$ boundary components, all covering non-contractible waists.\vspace{2pt}

\item[(ii)] \textnormal{\textbf{Completeness:}} $\Sigma$ intersects any orbit $\phi_{(-\infty,\infty)}(z)$ of the geodesic flow.\vspace{2pt}

\item[(iii)] \textnormal{\textbf{Escape set:}} There exists a possibly empty union of connected components $K\subset\partial\Sigma$, whose base projection $\pi(K)$ is the union of hyperbolic contractible waists, such that the complements $\Ws(K)\setminus K$ and $\Wu(K)\setminus K$ are given by
\begin{align*}
 \Ws(K)\setminus K &= \big\{ z\in SM\ \big|\  \phi_{[t,\infty)}(z)\cap\Sigma=\varnothing\mbox{ for some }t\in\R\big\},\\
  \Wu(K)\setminus K &= \big\{ z\in SM\ \big|\  \phi_{(-\infty,t]}(z)\cap\Sigma=\varnothing\mbox{ for some }t\in\R\big\}.
\end{align*}

\item[(iv)] \textnormal{\textbf{Return time:}} There exists $\ell>0$ such that, for each $z\in SM$ sufficiently close to $\partial\Sigma\setminus K$, we have $\phi_{(0,\ell]}(z)\cap\Sigma\neq\varnothing$.

\end{itemize}
\end{MainThm}

For general non-degenerate Reeb flows on closed 3-manifolds, surfaces of section satisfying properties analogous to~(ii,iii,iv) of Theorem \ref{mt:broken_book} were constructed by Colin--Dehornoy--Rechtman \cite{Colin:2020tl} by applying surgery to certain holomorphic curves provided by Hutchings' embedded contact homology \cite{Hutchings:2014vp}. Colin--Dehornoy--Rechtman further employed such surfaces of section to produce a so-called broken book decomposition of the 3-manifold, which is a generalization of the classical notions of open book decomposition and of Hofer--Wysocki--Zehnder's finite energy foliations \cite{Hofer:2003wf}.

The mentioned works \cite{Hofer:1998vy,Hofer:2003wf,Colin:2020tl,Contreras:2021vx,Colin:2022tq} are ultimately based on holomorphic curves techniques \cite{Hofer:2002vt}, and the last three ones even on embedded contact homology \cite{Hutchings:2014vp}. In this paper we do not need any of these techniques, and instead employ the curve shortening flow \cite{Gage:1990ws, Grayson:1989tj}. Our approach allows us to obtain sharper results for geodesic flows. On the one hand, our theorems only require the non-existence or the non-degeneracy of contractible simple closed geodesics without conjugate points, but no conditions on the other closed geodesics. On the other hand, in all the statements except Theorem~\ref{mt:Kupka_Smale} we obtain surfaces of section all of whose components have genus at most one, which may be important for future applications. 
Incidentally, our arguments require a study of maximal families of pairwise disjoint simple closed geodesics of closed orientable Riemannian surfaces (Theorems~\ref{t:scg1} and~\ref{t:scg2}), which may have independent interest. 

Finally, we remark that, even though we only considered Riemannian geodesic flows, our results are valid as well for geodesic flows of reversible Finsler metrics, by using the generalization of the curve shortening flow developed by Oaks \cite{Oaks:1994aa} and further investigated by Angenent \cite{Angenent:2008aa} and De Philippis et al.\ \cite{De-Philippis:2022wz}.

\subsection{Organization of the paper}
In Section~\ref{s:curve_shortening} we recall the main properties of the curve shortening flow, we provide the details of some applications to the existence of simple closed geodesics that cannot be found in the literature, and we recall some features of conjugate points. In Section~\ref{s:contractible_scg} we study configurations of simple closed geodesics on closed Riemannian surfaces, whose properties will be an essential ingredient for our main theorems. Finally, in Section~\ref{s:proofs}, we shall state and prove Theorems~\ref{mt:Birkhoff}, \ref{mt:Bangert}, \ref{mt:Kupka_Smale}, and~\ref{mt:broken_book}.

\subsection{Acknowledgements} We are grateful to Marie-Claude Arnaud for a discussion concerning the stable manifold of hyperbolic closed geodesics without conjugate points. 
We are also particularly grateful to the many anonymous referees for their careful reading of the manuscript, and their helpful reports.

\section{Preliminaries}
\label{s:curve_shortening}

\subsection{The curve shortening flow}\label{ss:curve_shortening}
It is well known that closed geodesics are critical points of the length and energy functionals, and therefore their existence can be investigated by means of critical point theory \cite{Klingenberg:1978wy}. In this paper, we will be interested in contractible closed geodesics on surfaces; their critical point theory requires the curve shortening flow \cite{Grayson:1989tj}, whose main properties we shall now recall.

Let $(M,g)$ be a closed oriented Riemannian surface. For any  smooth embedded circle $\gamma:S^1\hookrightarrow M$, we denote by $\nu_\gamma$ its positively oriented normal vector field, and by $k_\gamma$ the signed geodesic curvature of $\gamma$. Here, $S^1=\R/\Z$.
We denote by $\Emb(S^1,M)$ the space of smooth embedded circles in $M$ endowed with the $C^\infty$ topology. The length functional
\begin{align}
\label{e:length}
 L(\gamma)=\int_{S^1} \| \dot\gamma(t) \|_g\, dt
\end{align}
is continuous over this space. Indeed, it is even differentiable, and its critical points are the simple closed geodesics (that is, the closed geodesics in $\Emb(S^1,M)$) with arbitrary time reparametrization. The \emph{curve shortening flow} is a continuous map 
\[
\UU\to\Emb(S^1,M),
\qquad
(s,\gamma_0)\mapsto\Psi_s(\gamma_0):=\gamma_s,\] defined on a maximal open neighborhood $\UU\subset[0,\infty)\times\Emb(S^1,M)$ of $\{0\}\times\Emb(S^1,M)$ by means of the following PDE:
\begin{align*}
 \partial_s\gamma_s=k_{\gamma_s}\nu_{\gamma_s}.
\end{align*}
Its main properties are the following.
\begin{itemize}
\item[{(i)}] $\Psi_0=\id$, and $\Psi_{s_2}\circ\Psi_{s_1}=\Psi_{s_2+s_1}$ for all $s_1,s_2\geq0$;\vspace{3pt}

\item[{(ii)}] $\Psi_s(\gamma\circ\theta)=\Psi_s(\gamma)\circ\theta$ for all $(s,\gamma)\in\UU$ and $\theta\in\Diff(S^1)$;\vspace{3pt}

\item[{(iii)}] $\tfrac{d}{ds} L(\Psi_s(\gamma))\leq0$ for all $(s,\gamma)\in\UU$, and the equality holds if and only if $\gamma$ is a simple closed geodesic (not necessarily parametrized with constant speed);\vspace{3pt}

\item[{(iv)}] for each $\gamma\in\Emb(S^1,M)$, if $s_\gamma\in(0,\infty]$ denotes the supremum of the times $s>0$ for which $(s,\gamma)\in\UU$, then $s_\gamma$ is finite if and only if $\Psi_s(\gamma)$ converges to a constant as $s\to s_\gamma$.

\end{itemize}
We refer the reader to \cite{Grayson:1989tj,De-Philippis:2022wz} and to the references therein for the proofs of these facts. For each $\ell>0$ and $\epsilon>0$, we consider the open subsets
\begin{align*}
\Emb(S^1,M)^{<\ell} & := \big\{ \gamma\in\Emb(S^1,M)\ \big|\ L(\gamma)<\ell \big\},
\\
\WW(\ell,\epsilon) & := \big\{ \gamma\in\Emb(S^1,M)\ \big|\ |L(\gamma)-\ell|<\epsilon^2,\ \|k_\gamma\|_{L^\infty}<\epsilon \big\}.
\end{align*}
The intersection
\begin{align*}
\KK_\ell:=\bigcap_{\epsilon>0} \WW(\ell,\epsilon)
\end{align*}
is precisely the subspace of those embedded circles that are reparametrizations of simple closed geodesics of length $\ell>0$. Moreover, we have $\KK_\ell=\varnothing$ if and only if $\WW(\ell,\epsilon)=\varnothing$ for $\epsilon>0$ sufficiently small.

In order to employ the curve shortening flow in the critical point theory of the length functional, the following property is crucial. Its proof can be extracted from Grayson's \cite{Grayson:1989tj}, and the details can also be found in \cite[Th.~1.2(iv)]{De-Philippis:2022wz}.

\begin{itemize}

\item[{(v)}] For each $\ell>0$ and $\epsilon>0$ there exists $\delta\in(0,\ell)$ and a continuous function $\tau:\Emb(S^1,M)^{<\ell+\delta}\to[0,\infty)$ such that, for each $\gamma\in\Emb(S^1,M)^{<\ell+\delta}$, we have $\tau(\gamma)<s_\gamma$ and
\begin{align*}
\qquad\ \
 \Psi_s(\gamma)\in\Emb(S^1,M)^{<\ell-\delta}\cup\WW(\ell,\epsilon),\quad\forall  s\in[\tau(\gamma),s_\gamma).
\end{align*}

\end{itemize}

A path-connected subset $U\subseteq M$ is \emph{weakly convex} when, for any pair of distinct points $x,y\in U$ that can be joined by an absolutely continuous curve contained in $U$ of length strictly less than the injectivity radius $\inj(M,g)$, the shortest geodesic segment joining $x$ and $y$ is entirely contained in $U$. 
One of the crucial properties of the curve shortening flow is that it preserves the embeddedness of loops. This was first proved by Gage \cite[Sect.~3]{Gage:1990ws} as a consequence of a suitable maximum principle. The same arguments actually imply the following analogous property.
\begin{itemize}

\item[{(vi)}] The curve shortening flow preserves any weakly convex open subset $U\subseteq M$, i.e.~for any smooth embedded circle $\gamma$ in $U$ and for any $s\in(0,s_\gamma)$, the embedded circle $\Psi_s(\gamma)$ is still contained in $U$.

\end{itemize}

The following is a rather straightforward consequence of the properties of the curve shortening flow.
\begin{Lemma}
\label{l:convergence_scg}
Let $U\subseteq M$ be a weakly convex open subset that is not simply connected, and $\CC\subset\Emb(S^1,U)$ a connected component containing loops that are non-contractible in $U$. Then, there exists a sequence $\gamma_n\in\CC$ converging in the $C^2$-topology to a simple closed geodesic $\gamma$ contained in $\overline U$ of length
\begin{align*}
 L(\gamma)=\inf_{\zeta\in\CC} L(\zeta)>0.
\end{align*}
\end{Lemma}

\begin{Remark}
The simple closed geodesic $\gamma$ may not be a waist (see the definition in Section~\ref{ss:waist}) if it intersects the boundary $\partial U$.\hfill\qed
\end{Remark}

\begin{proof}[Proof of Lemma~\ref{l:convergence_scg}]
We claim that 
\begin{align*}
 \ell:=\inf_{\zeta\in\CC} L(\zeta)
 \geq
 2\,\inj(M,g)>0.
\end{align*}
Indeed, assume by contradiction that $L(\zeta)<2\,\inj(M,g)$ for some $\zeta\in\CC$. This implies that $\zeta$ can be written as 
\[\zeta(t)=\exp_{\zeta(0)}(V(t))\]
for some smooth function $V:S^1\to T_{\zeta(0)}M$. We define the smooth homotopy 
\begin{align*}
\zeta_r:S^1\to M,\  \zeta_r(t):=\exp_{\zeta(0)}(rV(t)),\qquad
r\in[0,1],
\end{align*}
which satisfies $\zeta_0\equiv\zeta(0)$ and $\zeta_1=\zeta$. Since the open subset $U$ is weakly convex, each $\zeta_r$ is contained in $U$. Therefore $\zeta$ is contractible in $U$, which is a contradiction.

We claim that 
\begin{align*}
\WW(\ell,\epsilon)\cap\CC\neq\varnothing,
\qquad
\forall \epsilon>0.
\end{align*}
Indeed, assume that $\WW(\ell,\epsilon)\cap\CC=\varnothing$ for some $\epsilon>0$.
Consider the constant $\delta>0$ and the continuous function $\tau:\Emb(S^1,M)^{<\ell+\delta}\to(0,\infty)$ provided by property~(v). For each $\zeta_0\in\CC$ of length $L(\zeta_0)<\ell+\delta$, property~(vi) guarantees that $\zeta_s:=\Psi_s(\zeta_0)\in\CC$ for all $s>0$ for which it is well defined. Property (v) then implies $L(\zeta_{\tau(\zeta)})<\ell-\delta$, contradicting the definition of $\ell$.

We choose an arbitrary $\gamma_n\in\WW(\ell,1/n)$ parametrized with constant speed. Up to extracting a subsequence, we have that $\gamma_n(0)\to x$ and $\dot\gamma_n(0)\to v$ for some $(x,v)\in TM$ with $\|v\|=\ell$. The curve $\gamma:S^1\to M$, $\gamma(t):=\exp_x(tv)$ is a closed geodesic of length $\ell$, and $\gamma_n\to\gamma$ in the $C^1$ topology. Since $\|k_{\gamma_n}-k_\gamma\|_{L^\infty}\leq 1/n$, 
 actually $\gamma_n\to\gamma$ in the $C^2$ topology. Finally, since the closed geodesic $\gamma$ is the $C^2$ limit of embedded circles on an orientable surface, $\gamma$ is a simple closed geodesic.
\end{proof}

We denote by $SM=\big\{v\in TM\ \big|\ \|v\|_g=1\big\}$ the unit tangent bundle, by $\pi:SM\to M$ the base projection, and by $\phi_t:SM\to SM$ the geodesic flow, which is defined by $\phi_t(\dot\gamma(0))=\dot\gamma(t)$ where $\gamma:\R\looparrowright M$ is a geodesic  parametrized with unit speed $\|\dot\gamma\|_g\equiv1$. In Section~\ref{s:contractible_scg}, we will need the following property of weakly convex sets.

\begin{Lemma}
\label{l:weakly_locally_convex}
If $U\subseteq M$ is a weakly convex subset, and $K\subset SM$ is a subset invariant under the geodesic flow $($i.e.~$\phi_t(K)=K$ for all $t\in\R$$)$, then any path-connected component of $U\setminus\pi(K)$ is weakly convex.
\end{Lemma}

\begin{proof}
Let $V\subset U\setminus\pi(K)$ be a path-connected component, and consider two arbitrary distinct points $x_1,x_2\in V$ that can be joined by an absolutely continuous curve $\zeta$ contained in $V$ of length strictly less than the injectivity radius $\inj(M,g)$. Since $U$ is weakly convex, the shortest geodesic segment $\gamma$ joining $x_1$ and $x_2$ is contained in $U$. Let us assume by contradiction that $\gamma\cap\pi(K)\neq\varnothing$, and choose a point $v\in K$ such that $x:=\pi(v)\in\gamma$. Let $B\subset M$ be the Riemannian open ball of radius $\inj(M,g)$ centered at $x$, and $\eta\subset \pi(K)\cap B$ the maximal geodesic segment passing through $x$ tangent to $v$. Notice that $\eta$ separates $B$, and intersects $\gamma$ only in $x$. In particular, $x_1$ and $x_2$ lie in distinct connected components of $B\setminus\eta$. This is not possible, since the curve $\zeta$ joining $x_1$ and $x_2$ is contained in $B\setminus\pi(K)\subset B\setminus\eta$; indeed, any absolutely continuous curve joining $x_1$ and $x_2$ and not entirely contained in $B$ must have length larger than or equal to $\inj(M,g)$.
\end{proof}

\subsection{Types of closed geodesics}\label{ss:waist}
Let $\gamma$ be a closed geodesic of length $\ell>0$ in the closed oriented Riemannian surface $(M,g)$. We parametrize $\gamma$ with unit speed, so that it is a curve of the form 
\[\gamma:\R/\ell\Z\to M,\qquad \gamma(t)=\pi\circ\phi_t(v),\] 
where $v\in SM$, and the length $\ell$ is the minimal period of $\gamma$.
The Floquet multipliers of $\gamma$ are the eigenvalues of the linearized Poincaré map $d\phi_\ell(v)|_V$, where $V$ is the vector subspace
\begin{align*}
 V := \big\{ w\in T_{v}SM\ \big|\ g(v,d\pi(v)w)=0\big\}.
\end{align*}
Since the linearized Poincar\'e map preserves a symplectic structure on the plane $V$, the Floquet multipliers are of the form $\sigma,\sigma^{-1}\in \U\cup\R\setminus\{0\}$, where $\U$ denotes the unit circle in the complex plane. The closed geodesic $\gamma$ is called \emph{non-degenerate}\footnote{A closed geodesic $\gamma$ of length $\ell$ and Floquet multiplier $\sigma=e^{i2\pi/k}$ for some integer $k\geq2$ is non-degenerate. However, the $k$-th iterate of $\gamma$ is degenerate. In this paper, we will not need to consider iterates of closed geodesics.} when $\sigma\neq 1$, and \emph{hyperbolic} when $\sigma\in\R\setminus\{1,-1\}$. 

We recall that a simple closed geodesic $\gamma:S^1\hookrightarrow M$ is called a \emph{waist} when any absolutely continuous curve $\zeta:S^1\to M$ that is sufficiently $C^0$-close to $\gamma$ satisfies $L(\zeta)\geq L(\gamma)$; here $L$ is the length functional~\eqref{e:length}. 

\begin{Remark}\label{r:strict_waists}
Let $\gamma:S^1\hookrightarrow M$ be a waist,  $W\subset M$ a sufficiently small open neighborhood of $\gamma$, and $\WW$ the space of absolutely continuous curves $\zeta:S^1\to W$ homotopic to $\gamma$ within $W$. One can easily prove that any $\zeta\in\WW$ has length $L(\zeta)\geq L(\gamma)$. Moreover, the equality $L(\zeta)= L(\gamma)$ holds if and only if $\zeta$ becomes itself a waist after being reparametrized with constant speed $\|\dot\zeta\|_g\equiv L(\zeta)$. 
If $\gamma$ is a non-degenerate waist, then any $\zeta\in\WW$ satisfying the equality $L(\zeta)=L(\gamma)$ must be geometrically equivalent to $\gamma$, i.e.~$\zeta\circ\theta=\gamma$ for some homeomorphism $\theta:S^1\to S^1$.
\hfill\qed
\end{Remark}

Later on, we shall employ waists to infer the existence of other simple closed geodesics, according to the following lemma. For closed geodesics that are possibly self-intersecting, the analogous lemma is well known, see e.g.~\cite{Bangert:1993ue}.

\begin{Lemma}
\label{l:minmax}
Let $(M,g)$ be an oriented Riemannian surface.
\begin{itemize}
\item[(i)] If $A\subset M$ is a compact annulus bounded by two waists, then $\interior(A)$ contains a non-contractible simple closed geodesic.\vspace{3pt}

\item[(ii)] If $D$ is a compact disk bounded by a waist, then $\interior(D)$ contains a simple closed geodesic.

\end{itemize}
\end{Lemma}

\begin{proof}
Let $A\subset M$ be a compact annulus bounded by the waists $\gamma_0,\gamma_1$, with $L(\gamma_1)\leq L(\gamma_0)$. Let $W\subset A\setminus\gamma_1$ be a sufficiently small open neighborhood of $\gamma_0$.  If $W\setminus\gamma_0$ contains a simple closed geodesic, such a closed geodesic must be $C^0$-close to $\gamma_0$ and in particular non-contractible in $A$, and we are done. Assume now that $W\setminus\gamma_0$ does not contain any simple closed geodesic. In particular, every absolutely continuous curve $\gamma:S^1\to W$ homotopic to $\gamma_0$ within $W$ and not geometrically equivalent to $\gamma_0$ satisfies the strict inequality $L(\gamma)> L(\gamma_0)$ (see Remark~\ref{r:strict_waists}). Up to  shrinking $W$, there exists $\rho>0$ such that any smooth curve $\gamma:S^1\hookrightarrow \overline W$ homotopic to $\gamma_0$ within $\overline W$ and that intersects $\partial W\setminus\gamma_0$ has length $L(\gamma)\geq L(\gamma_0)+\rho$. We now detect a simple closed geodesic in $\interior(A)$ by means of a minmax procedure, as follows: let
\begin{align}
\label{e:minmax}
 \ell:=\inf_{\FF} \max_{r\in[0,1]} L(\zeta_r),
\end{align}
where the infimum ranges over the family $\FF$ of continuous homotopies of smooth embedded loops $\zeta_r:S^1\hookrightarrow A$, $r\in[0,1]$, such that $\zeta_0=\gamma_0$, $\zeta_1=\gamma_1$, and $\zeta_r\subset\interior(A)$ for all $r\in(0,1)$. Notice that, for any such homotopy $(\zeta_r)_{r\in[0,1]}$, there exists a minimal $r_0\in(0,1]$ such that the curve $\zeta_r$ intersects $\partial W\setminus\gamma_0$, and therefore \[L(\zeta_{r_0})\geq L(\gamma_0)+\rho.\] 
This readily implies that
\begin{align*}
 \ell \geq L(\gamma_0)+\rho>L(\gamma_0).
\end{align*}
We claim that $\ell$ is the length of a non-contractible simple closed geodesic in $A$, which must actually be contained in $\interior(A)$ since $L(\gamma_1)\leq L(\gamma_0)<\ell-\rho$. Let us assume that this is not the case. In particular, for $\epsilon>0$ small enough, no embedded loop $\gamma\in\WW(\ell,\epsilon)$ is entirely contained in $A$ and non-contractible therein. We apply property~(v) of the curve shortening flow, which provides $\delta>0$ and a suitable continuous function $\tau:\Emb(S^1,M)^{<\ell+\delta}\to(0,\infty)$. Let $(\zeta_r)_{r\in[0,1]}$ be a homotopy in $\FF$ that is optimal up to $\delta$, meaning that 
\[\max_{r\in[0,1]} L(\zeta_r)< \ell+\delta.\]
We push the homotopy by means of the curve shortening flow, defining
\begin{align*}
 \eta_r:=\Psi_{\tau(\zeta_r)}(\zeta_r),\ r\in[0,1].
\end{align*}
Property~(v) implies that $\eta_r\in\Emb(S^1,M)^{<\ell-\delta}\cup\WW(\ell,\epsilon)$ for each $r\in[0,1]$. Since the compact annulus $A$ is bounded by simple closed geodesics, $\interior(A)$ is weakly convex. 
Property~(vi) implies that $\eta_r\subset \interior(A)$ for each $r\in(0,1)$. Since $\zeta_0=\gamma_0$ and $\zeta_1=\gamma_1$ are simple closed geodesics, we have $\eta_0=\gamma_0$ and $\eta_1=\gamma_1$. 
Therefore the homotopy $(\eta_r)_{r\in[0,1]}$ belongs to $\FF$. Since no curve $\gamma\in\WW(\ell,\epsilon)$ is entirely contained in $A$ and non-contractible therein, we conclude that  $\eta_r\in\Emb(S^1,M)^{<\ell-\delta}$ for each $r\in[0,1]$, contradicting the definition of the minmax value $\ell$.

The case of a compact disk $D\subset M$ bounded by a waist $\gamma$ is analogous, except that in the definition of the minmax~\eqref{e:minmax} the infimum ranges over the family of continuous homotopies of smooth embedded loops $\zeta_r:S^1\hookrightarrow D$, $r\in[0,1]$, such that $\zeta_0=\gamma$,  $\zeta_r\subset\interior(D)$ for all $r\in(0,1]$, and $L(\zeta_1)<L(\gamma)$.
\end{proof}

\subsection{Conjugate points}\label{ss:conjugate_points}
Let us recall the classical notion of conjugate points. If $\gamma:\R\to M$ is an (open or closed) geodesic parametrized with unit speed, the points $\gamma(t_1),\gamma(t_2)$ are conjugate along $\gamma|_{[t_1,t_2]}$ when 
\begin{align*}
\ker d(\pi\circ\phi_{t_2-t_1}(\dot\gamma(t_1)))|_{\ker(d\pi(\dot\gamma(t_1)))}\neq\{0\}.
\end{align*}

On orientable Riemannian surfaces, simple closed geodesics with conjugate points are not waists, and actually satisfy the following lemma due to Bangert \cite[Lemma~2]{Bangert:1993ue}.

\begin{Lemma}
\label{l:Bangert}
Let $(M,g)$ be an orientable Riemannian surface, and $\gamma:S^1\hookrightarrow M$ a simple closed geodesic with conjugate points. Then, for any open neighborhood $U\subset M$ of $\gamma$, any connected component $V\subset U\setminus\gamma$ contains a smooth embedded circle $\zeta:S^1\hookrightarrow V$ homotopic to $\gamma$ within $U$ such that $L(\zeta)<L(\gamma)$.
\hfill\qed
\end{Lemma}

On an orientable Riemannian surface, a non-degenerate simple closed geodesic is a waist if and only if it does not have conjugate points, see for instance \cite[Lemma~4.1(iii) and Prop.~4.2(iii,vii)]{De-Philippis:2022wz}. Moreover, a non-degenerate waist $\gamma$ is hyperbolic, see for instance \cite[Theorem~3.4.2]{Klingenberg:1995aa}, and the corresponding orbit $\dot\gamma$ of the geodesic flow $\phi_t$ has a stable manifold
\begin{align*}
 \Ws(\dot\gamma)=\big\{ z\in SM\ \big|\ \omega(z)=\dot\gamma\big\},
\end{align*}
which is an injectively immersed surface in $SM$. Here, $\omega(z)$ denotes the $\omega$-limit of $z$, i.e.
\begin{align*}
 \omega(z)=\bigcap_{t>0}\overline{\phi_{[t,\infty)}(z)}.
\end{align*}

\begin{Lemma}\label{l:Ws_waist}
Let $(M,g)$ be an orientable closed Riemannian surface, and $\gamma$ a non-de\-gen\-er\-ate waist. For any sufficiently small neighborhood $V\subset SM$ of $\dot\gamma$, if $W\subset V\cap\Ws(\dot\gamma)$ is the path-connected component containing $\dot\gamma$, the base projection $\pi|_W:W\to M$ is a diffeomorphism onto a neighborhood of $\gamma$. In particular, for each $z\in \Ws(\dot\gamma)\setminus\dot\gamma$ with associated geodesic $\zeta(t):=\pi\circ\phi_t(z)$, there exists $t\in\R$ such that $\zeta|_{[t,\infty)}$ does not intersect~$\gamma$.
\end{Lemma}

\begin{Remark}
Since $\Wu(\dot\gamma)=-\Ws(-\dot\gamma)$, the analogous statement holds for the unstable manifold.\hfill\qed
\end{Remark}

\begin{proof}[Proof of Lemma~\ref{l:Ws_waist}]
Since $\gamma$ is hyperbolic, the closed orbit $\dot\gamma$ has a Floquet multiplier $\sigma\in(-1,1)$. We parametrize $\gamma$ with unit speed, and denote by  $\ell>0$ its length, so that $\dot\gamma(t)=\dot\gamma(t+\ell)$. We consider the stable line bundle $\Es$ over $\dot\gamma$, which is given by
\begin{align*}
\Es(\dot\gamma(t))=\ker\big(d\phi_\ell(\dot\gamma(t))-\sigma\,\id\big).
\end{align*}
The stable bundle is invariant under the linearized geodesic flow, i.e.\ 
\[
d\phi_t(\dot\gamma(0))\Es(\dot\gamma(0))=\Es(\dot\gamma(t)).
\]
If we had $\Es(\dot\gamma(t))=\ker(d\pi(\dot\gamma(t)))$ for some $t\in\R/\ell\Z$, the endpoints of the geodesic segment $\gamma|_{[t,t+\ell]}$ would be conjugate.
Therefore, since $\gamma$ has no conjugate points, we infer
\begin{align*}
 \Es(\dot\gamma(t))\cap\ker(d\pi(\dot\gamma(t)))=\{0\},\qquad\forall t\in\R/\ell\Z.
\end{align*}
This, together with the fact that 
\[
T_{\dot\gamma(t)}\Ws(\dot\gamma)=\mathrm{span}\{\dot\gamma(t)\}\oplus\Es(\dot\gamma(t)),\qquad\forall t\in\R/\ell\Z,
\]
readily implies that the local stable manifold of $\dot\gamma$ is a graph over the base manifold $M$. Namely, for any sufficiently small neighborhood $V\subset SM$ of the closed orbit $\dot\gamma$, if $W\subset V\cap\Ws(\dot\gamma)$ is the path-connected component containing $\dot\gamma$, the base projection $\pi|_W:W\to M$ is a diffeomorphism onto a neighborhood of $\gamma$. For each $z\in \Ws(\dot\gamma)\setminus\dot\gamma$, for all $t>0$ large enough we have $\phi_t(z)\in W\setminus\dot\gamma$, and therefore $\pi\circ\phi_t(z)\in\pi(W)\setminus\gamma$.
\end{proof}

We recall the following elementary property of geodesics with conjugate points on surfaces (see for instance \cite[Lemma~5.9]{De-Philippis:2022wz} for a proof).

\begin{Lemma}\label{l:intersecting_conj_pts}
Let $(M,g)$ be a Riemannian surface, and $\gamma:[-T,T]\to M$ a geodesic arc such that, for some $[t_1,t_2]\subset(-T,T)$, the points $\gamma(t_1)$ and $\gamma(t_2)$ are conjugate along $\gamma|_{[t_1,t_2]}$. There exists an open neighborhood $V\subset SM$ of $\dot\gamma(0)$ such that, for each $z\in V$, the geodesic $\zeta(t):=\pi\circ\phi_t(z)$ intersects $\gamma$ for some $t\in[-T,T]$.\hfill\qed
\end{Lemma}

This lemma has the following immediate consequence for simple closed geodesics with conjugate points.

\begin{Cor}
\label{c:intersecting}
Let $(M,g)$ be an orientable Riemannian surface, and $\gamma$ a simple closed geodesic with conjugate points. 
\begin{itemize}
\item[(i)] There exists $T>0$ and an open neighborhood $V\subset SM$ of the lift $\dot\gamma$ such that, for each $z\in V$, the geodesic $\zeta(t):=\pi\circ\phi_t(z)$ intersects $\gamma$ for some positive time $t\in(0,T]$ and for some negative time $t\in[-T,0)$.\vspace{5pt}

\item[(ii)] There exists $T>0$ and an open neighborhood $U\subset M$ of $\gamma$ such that, for each $z\in SU$, the geodesic $\zeta(t):=\pi\circ\phi_t(z)$ intersects $\gamma$ for some $t\in[-T,T]$.
\end{itemize}
\hfill\qed
\end{Cor}

\section{Contractible simple closed geodesics on surfaces}
\label{s:contractible_scg}

Let $(M,g)$ be a closed Riemannian surface with geodesic flow $\phi_t:SM\to SM$. If $K\subset SM$ is a hyperbolic compact invariant subset for the geodesic flow, its stable manifold is defined by
\begin{align*}
 \Ws(K)=\big\{ z\in SM\ \big|\ \omega(z)\subseteq K \big\},
\end{align*}
where $\omega(z)$ denotes the $\omega$-limit of $z$. The unstable manifold $\Wu(K)$ is defined analogously by employing the $\alpha$-limit instead of the $\omega$-limit, or equivalently
\begin{align*}
 \Wu(K)=-\Ws(-K).
\end{align*}
In this paper, we will consider invariant compact sets $K$ of the form
\begin{align*}
K=\bigcup_{i=1}^n (\dot\gamma_i\cup-\dot\gamma_i), 
\end{align*}
where $\gamma_1,...,\gamma_n$ are hyperbolic closed geodesics of $(M,g)$. In this case, the stable and unstable manifolds of $K$ decompose as the disjoint unions of the stable and unstable manifolds of the closed orbits $\pm\dot\gamma_i$, i.e.
\begin{align*}
\Ws(K)=\bigcup_{i=1}^n \big(\Ws(\dot\gamma_i)\cup\Ws(-\dot\gamma_i)\big),
\qquad
 \Wu(K)=\bigcup_{i=1}^n \big(\Wu(\dot\gamma_i)\cup\Wu(-\dot\gamma_i)\big).
\end{align*}

Let $V\subset SM$ be an open subset. We define the \emph{forward trapped set} $\trap_+(V)$ and the \emph{backward trapped set} $\trap_-(V)$ by 
\begin{align*}
 \trap_\pm(V) = \big\{z\in SM\ \big|\ \phi_{\pm t}(z)\in V\mbox{ for all $t>0$ large enough}\big\}.
\end{align*}
Notice that
\begin{align}
\label{e:trap+-}
 \trap_-(V)=-\trap_+(-V).
\end{align}

In Section~\ref{ss:curve_shortening} we introduced the notion of weak convexity for an open subset of a closed orientable Riemannian surface $(M,g)$. In this section, we will consider open disks $B\subset M$ satisfying the following, stronger, convexity assumption.

\begin{Assumption}
\label{a:convexity}
The open disk $B\subset M$ is weakly convex, and there exist $\delta>0$, $T>0$, and an open neighborhood $N\subset\overline B$ of $\partial B$ such that every smooth curve $\gamma:[-T,T]\to M$ parametrized with unit speed $\|\dot\gamma\|_g\equiv1$, with curvature bound $\|k_\gamma\|_{L^\infty}\leq\delta$, and such that $\gamma(0)\in N$, is not entirely contained in $B$.\hfill\qed
\end{Assumption}

\begin{figure}
\includegraphics{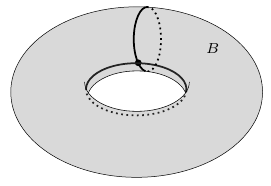}
\caption{An open convex geodesic polygon $B$ that is the complement of two simple closed geodesics of a 2-torus of revolution.}
\label{f:torus}
\end{figure}

\begin{Example}\label{ex:polygon}
Assumption~\ref{a:convexity} is satisfied when $B\subset M$ is a \emph{convex geodesic polygon}, meaning an open disk whose boundary $\partial B$, seen as a piecewise smooth immersed submanifold of $M$, is a piecewise geodesic circle with at least one corner, and all the inner angles at the corners of $\partial B$ are less than $\pi$. 
This definition allows $B$ to be a fundamental domain of $M$ defined as the complement of a suitable collection of finitely many simple closed geodesics (Figure~\ref{f:torus}).\hfill\qed
\end{Example}

\begin{Example}
\label{ex:conjugate_points}
Corollary~\ref{c:intersecting}(ii) readily implies that Assumption~\ref{a:convexity} is satisfied when $B\subset M$ is an open disk whose boundary is a simple closed geodesic with conjugate points.\hfill\qed
\end{Example}

The following two theorems are the main results of this section.

\begin{Thm}\label{t:scg1}
Let $(M,g)$ be a closed orientable Riemannian surface, and $B\subset M$ an open disk satisfying Assumption~\ref{a:convexity}. If $B$ does not contain any simple closed geodesic without conjugate points, then $\trap_+(SB)=\trap_-(SB)=\varnothing$.
\end{Thm}

\begin{proof}
Equation~\eqref{e:trap+-} implies that 
$\trap_+(SB)=-\trap_-(SB)$. 
Let us assume that $\trap_+(SB)\neq\varnothing$, and consider an arbitrary point $z\in \trap_+(SB)$.  Assumption~\ref{a:convexity} implies
\begin{align*}
 \pi(\omega(z))\subset B\setminus N,\qquad\forall z\in\trap_+(SB).
\end{align*}
By Lemma~\ref{l:weakly_locally_convex}, the connected component $U\subset B\setminus\pi(\omega(z))$ containing $N\setminus\partial B$ is weakly convex. Let $\CC\subset\Emb(S^1,U)$ be a connected component containing loops that are non-contractible in $U$. By Lemma~\ref{l:convergence_scg}, there exists a sequence $\gamma_n\in\CC$ that converges in the $C^2$ topology to a simple closed geodesic $\gamma\subset \overline U$ of length
\begin{align}
\label{e:inf_waist_lemma1}
 L(\gamma)=\inf_{\zeta\in\CC} L(\zeta).
\end{align}
Notice that $\gamma_n\subset U\setminus N$ for all $n$ large enough, for otherwise $\gamma_n$ would intersect $\partial B$ according to Assumption~\ref{a:convexity}. Therefore $\gamma$ is contained in $B$. This, together with Lemma~\ref{l:Bangert} and Equation~\eqref{e:inf_waist_lemma1}, implies that $\gamma$ has no conjugate points.
\end{proof}

\begin{Thm}\label{t:scg2}
Let $(M,g)$ be a closed orientable Riemannian surface, and $B\subset M$ an open disk satisfying Assumption~\ref{a:convexity}, containing at least one closed geodesic, but no degenerate simple closed geodesics without conjugate points. Then there exists a finite even number of pairwise disjoint simple closed geodesics 
$\gamma_1,...,\gamma_{2k}\subset B$ satisfying the following properties:
\begin{itemize}

\item[(i)] $\gamma_{i+1}\subset B_{\gamma_i}$, where $B_{\gamma_i}\subset B$ are the open disks with boundary $\partial B_{\gamma_i}=\gamma_i$.\vspace{2pt}

\item[(ii)] For each $i$ odd, $\gamma_i$ is a waist.\vspace{2pt}

\item[(iii)] For each $i$ even, $\gamma_i$ has conjugate points.\vspace{2pt}

\item[(iv)] Let $U:=B\setminus(\gamma_1\cup...\cup\gamma_{2k})$. The trapped sets of $SU$ are given by
\begin{align*}
\qquad\quad
\trap_+(SU) = \Ws(K)\setminus K,
\qquad
\trap_-(SU) = \Wu(K)\setminus K,
\end{align*}
where
\begin{align*}
 K:=\bigcup_{i\ \mathrm{odd}} (\dot\gamma_i\cup-\dot\gamma_i).
\end{align*}

\item[(v)] No complete geodesic $($i.e.~a geodesic parametrized with constant speed and defined for all times$)$ is entirely contained in $U$.

\end{itemize}
\end{Thm}

We shall prove Theorem~\ref{t:scg2} after some preliminaries.

\begin{Lemma}
\label{l:annulus}
Let $(M,g)$ be a closed oriented Riemannian surface. There exists a constant $a_0>0$ such that, for each embedded compact annulus $A\subset M$ with $\area(A,g)\leq a_0$ and whose boundary $\partial A$ is the union of two simple closed geodesics $\gamma_1,\gamma_2$, we have 
\[
\tfrac29 L(\gamma_1)\leq
L(\gamma_2)\leq \tfrac92 L(\gamma_1).\]
\end{Lemma}

\begin{proof}
We denote by $J$ the complex structure of the oriented Riemannian surface $(M,g)$, i.e.\ for all tangent vectors $v\in S_xM$ we have that $v,Jv$ is an oriented orthonormal basis of $T_xM$. For each $r\in(0,\inj(M,g)/2)$, $x\in M$, and $v\in S_xM$, we denote by $T(v,r)$ the unique geodesic triangle with vertices $x,\exp_x(rJv),\exp_x(-\tfrac r2 v)$, see Figure~\ref{f:triangle}. We consider the continuous function
\begin{align*}
 a:(0,\inj(M,g)/4)\to(0,\infty),\qquad
 a(r)=\min_{v\in SM}\area(T(v,r),g).
\end{align*}
We fix $r\in(0,\inj(M,g)/4)$ to be small enough so that, for each $x\in M$ and $v\in S_xM$ with corresponding geodesic $\gamma_v(t):=\exp_{x}(tv)$, the smooth map
\begin{align}
\label{e:no_focal_points}
[0,r]\times[0,r]\to M,\qquad (t,s)\mapsto \exp_{\gamma_v(t)} (s J\dot\gamma_v(t)).
\end{align}
is an embedding.
The positive constant of the lemma will be $a_0:=a(r)$. 

\begin{figure}
\includegraphics{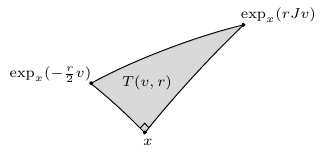}
\caption{The geodesic triangle $T(v,r)$.}
\label{f:triangle}
\end{figure}

Consider two simple closed geodesics $\gamma_1,\gamma_2$ bounding a compact annulus $A\subset M$ such that $\area(A,g)\leq a(r)$.
We choose two points $x_0\in\gamma_1$ and $y_0\in\gamma_2$ such that
\begin{align}
\label{e:closest_points}
 d_g(x_0,y_0)=\min \big\{ d_g(x,y)\ \big|\ x\in\gamma_1,\ y\in\gamma_2 \big\}.
\end{align}
We set $\ell_1:=L(\gamma_1)$, $\ell_2:=L(\gamma_2)$, and we consider the unit-speed parametrizations 
\[\gamma_1:\R/\ell_1\Z\hookrightarrow M,\qquad \gamma_2:\R/\ell_2\Z\hookrightarrow M\] 
such that $\gamma_1(0)=x_0$, $\gamma_2(0)=y_0$, $J\dot\gamma_1(0)$ points inside the annulus $A$, while $J\dot\gamma_2(0)$ points outside the annulus $A$. 
Notice that \[d_g(x_0,y_0)<r.\] Otherwise, we would have $T(\dot\gamma_1(0),r)\subset A$, which would give the contradiction 
\[ a(r)\geq\area(A,g)> \area(T(\dot\gamma_1(0),r))\geq  a(r). \]
For each $t\in\R/\ell_1\Z$, we consider the geodesic arc 
\[\zeta_t:[0,r]\to M,\qquad \zeta_t(s)=\exp_{\gamma_1(t)}(s\,J\dot\gamma_1(t)).\]
The geodesic $\gamma_2$ intersects $\zeta_0((0,r))$ at $y_0=\gamma_2(0)$. Notice that all intersections between $\gamma_2$ and the geodesic arcs $\zeta_t$ must be transverse (otherwise some $\zeta_t$ would be contained in $\gamma_2$, contradicting the fact that $\gamma_2$ and $\gamma_1$ are disjoint).
 Therefore, by the implicit function theorem, there exists a maximal time $t_0\in(0,\ell_1]$ and a smooth monotone increasing function $\sigma:[0,t_0]\to\R$ such that $\sigma(0)=0$ and, for all $t\in [0,t_0]$, $\gamma_2$ intersects $\zeta_t((0,r])$ transversely at $\gamma_2(\sigma(t))$.
We denote by $\rho:[0,t_0]\to(0,r]$ the smooth function such that $\zeta_t(\rho(t))=\gamma_2(\sigma(t))$ for all $t\in [0,t_0]$. Notice that 
\begin{align}
\label{e:intersection_zetat_gamma2}
 \zeta_t|_{[0,\rho(t))} \cap \gamma_2 = \varnothing,\qquad
 \forall t\in[0,t_0].
\end{align}

We claim that $t_0=\ell_1$. Indeed, if $t_0<\ell_1$,
we would have $\gamma_2(\sigma(t_0))=\zeta_{t_0}(r)$. Therefore $T(\dot\gamma_1(t_0),r)\subset A$ (Figure~\ref{f:triangle2}), which would give the contradiction
\begin{align*}
\area(A,g)>\area(T(\dot\gamma_1(t_0),r))\geq a(r)\geq\area(A,g).
\end{align*}

\begin{figure}
\includegraphics{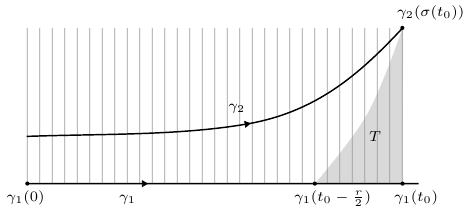}
\caption{The geodesic triangle $T=T(\dot\gamma_1(t_0),r)$ contained in $A$.}
\label{f:triangle2}
\end{figure}

Next, we claim that $\sigma(\ell_1)\leq\ell_2$. Indeed, assume by contradiction that $\sigma(\ell_1)>\ell_2$, so that there exists $t_1\in(0,\ell_1)$ such that $\sigma(t_1)=\ell_2$. Suitable segments of the geodesics $\gamma_1$, $\zeta_{t_1}$, and $\zeta_0$ bound a triangle $T'\subset A$ as in Figure~\ref{f:triangle3}. By~\eqref{e:intersection_zetat_gamma2}, $\gamma_2$ does not intersect the interior of $T'$.
Since the maps~\eqref{e:no_focal_points} are embeddings, the geodesic segments $\zeta_t$ and $\zeta_{t'}$ are disjoint whenever $|t-t'|<r$. Therefore, for each $t\in(t_1,\ell_1)$ the geodesic segment $\zeta_t$ must intersect $\zeta_{\ell_1}=\zeta_0$, but when $t>\ell_1-r$ this contradicts the embeddedness of the maps~\eqref{e:no_focal_points}.

\begin{figure}
\includegraphics{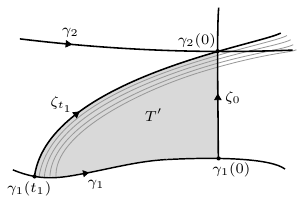}
\caption{The geodesic triangle $T'$.}
\label{f:triangle3}
\end{figure}

Finally, we claim that $\sigma(\ell_1)=\ell_2$. Indeed, assume by contradiction that the strict inequality $\sigma(\ell_1)<\ell_2$ holds, so that we have distinct values $\rho(0)\neq\rho(\ell_1)$. Since $\gamma_2(0)=\zeta_0(\rho(0))$ and $\gamma_2(\sigma(\ell_1))=\zeta_0(\rho(\ell_1))$, Equation~\eqref{e:closest_points} implies $\rho(0)<\rho(\ell_1)$. Namely, $\gamma_2(0)$ lies in $\zeta_{0}|_{[0,\rho(\ell_1))}=\zeta_{\ell_1}|_{[0,\rho(\ell_1))}$ (Figure~\ref{f:trapped}), which contradicts~\eqref{e:intersection_zetat_gamma2}.

\begin{figure}
\includegraphics{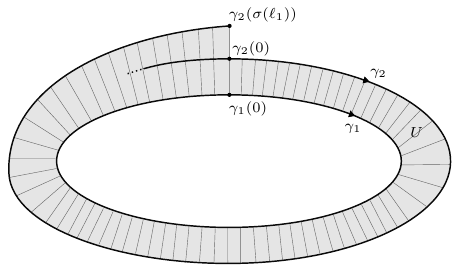}
\caption{The point $\gamma_2(0)$ intersecting the geodesic arc $\zeta_{\ell_1}|_{[0,\rho(\ell_1))}$.}
\label{f:trapped}
\end{figure}

We can now estimate the length of $\gamma_2$ as follows. Notice that 
\[\ell_1\geq2\,\inj(M,g)>8r.\] 
Let $k\geq9$ be the integer such that 
\[
\frac{\ell_1}{k}<r\leq\frac{\ell_1}{k-1}.
\]
Notice that each restriction $\gamma_2|_{[\sigma(t),\sigma(t+\ell_1/k)]}$ is contained in the open Riemannian ball $B_g(\gamma_1(t),2r)$ which has diameter less than $\inj(M,g)$. Therefore
\begin{align*}
 L(\gamma_2|_{[\sigma(t),\sigma(t+\ell_1/k)]})
 =
 d_g(\gamma_2(\sigma(t)),\gamma_2(\sigma(t+\ell_1/k)))
 <
 4r,
 \qquad\forall t\in[0,\ell_1].
\end{align*}
We infer
\begin{align*}
\ell_2 
&=
\sum_{h=0}^{k-1} d_g\big(\gamma_2(\sigma(\ell_1 \tfrac hk)),\gamma_2(\sigma(\ell_1\tfrac{h+1}k))\big)
<4rk
<4\ell_1\tfrac{k}{k-1}
\leq
\tfrac92 \ell_1.
\end{align*}
Switching the roles of $\gamma_1$ and $\gamma_2$, we obtain the other inequality $\ell_1\leq \tfrac92\ell_2$.
\end{proof}

We now consider an open disk $B\subset M$ satisfying Assumption~\ref{a:convexity}.  
 For any smooth embedded circle $\gamma \subset B$, we denote by $B_\gamma\subset B$ its filling, namely the open disk with boundary $\partial B_\gamma=\gamma$.

\begin{Lemma}
\label{l:gamma_0}
If $B$ contains closed geodesics, then it contains a simple closed geodesic $\gamma$ without conjugate points and whose filling $B_{\gamma}$ contains any other closed geodesic in $B$.
\end{Lemma}

\begin{proof}
We consider the compact subset
\begin{align*}
 K:=\overline{\bigcup_{\gamma} \dot\gamma},
\end{align*}
where the union ranges over all closed geodesics $\gamma$ contained in $B$. The subset $K$ is invariant under the geodesic flow (i.e.\ $\phi_t(K)=K$ for all $t\in\R$), and therefore its projection $\pi(K)$ is contained in the compact subset $B\setminus N$, where $N\subset\overline B$ is the open subset of $\partial B$ given by Assumption~\ref{a:convexity}. By Lemma~\ref{l:weakly_locally_convex}, the connected component $U\subset B\setminus\pi(K)$ intersecting $N$ is weakly convex. Since $U$ is not simply connected, there exists a connected component $\CC\subset\Emb(S^1,U)$ consisting of loops that are non-contractible in $U$ and whose fillings  contain $\pi(K)$. By Lemma~\ref{l:convergence_scg}, there exists a sequence $\gamma_n\in\CC$ converging in the $C^2$ topology to a simple closed geodesic $\gamma\subseteq \overline U$ of length
\begin{align}
\label{e:inf_waist_lemma2}
 L(\gamma)=\inf_{\zeta\in\CC} L(\zeta),
\end{align}
and whose filling $B_\gamma$ contains $\pi(K)\setminus\gamma$.
We now conclude as in the proof of Theorem~\ref{t:scg1}. We must have $\gamma_n\subset U\setminus N$ for all $n$ large enough, for otherwise $\gamma_n$ would intersect $\partial B$ according to Assumption~\ref{a:convexity}. Therefore $\gamma$ is contained in $B$. This, together with Lemma~\ref{l:Bangert} and Equation~\eqref{e:inf_waist_lemma2}, implies that $\gamma$ has no conjugate points.
\end{proof}

From now on, we assume that $B$ contains closed geodesics, and  that all simple closed geodesics without conjugate points and entirely contained in $B$ are non-degenerate (and therefore they are hyperbolic waists). The open ball $B$ may contain infinitely many simple closed geodesics. Nevertheless, we have the following statement.

\begin{Lemma}\label{l:V_finite}
Any collection of pairwise disjoint simple closed geodesics contained in $B$ is finite.
\end{Lemma}
\begin{proof}
Let $\GG$ be a collection of pairwise disjoint simple closed geodesics contained in the open ball $B$. Let $\GG'\subseteq\GG$ be a maximal subcollection of simple closed geodesics such that $B_{\gamma_1}\cap B_{\gamma_2}=\varnothing$ for all distinct $\gamma_1,\gamma_2\in\GG'$; namely, for each $\zeta\in\GG\setminus\GG'$, we have $B_{\zeta}\cap B_{\gamma}\neq\varnothing$ for some $\gamma\in\GG'$. 
Gauss-Bonnet formula implies
\begin{align*}
2\pi=\int_{B_{\gamma}} K_g\,dm_g\leq \max(K_g)\,\area(B_{\gamma},g),
\qquad\forall \gamma\in\GG,
\end{align*}
where $m_g$ is the Riemannian measure, and $K_g$ the Gaussian curvature of $(M,g)$. Therefore, $\GG'$ is a finite collection of cardinality
\begin{align*}
k:=\#\GG'\leq \frac{\max(K_g)\,\area(B,g)}{2\pi}.
\end{align*}
We define $U\subset B$ to be the complement of the simple closed geodesics in $\GG'$, i.e.
\begin{align*}
U:=B\setminus\bigcup_{\gamma\in\GG'}\gamma.
\end{align*}

We consider the family $\pi_0(\Emb(S^1,U))$ of path-connected components of the space of  embedded loops in $U$. Notice that $\pi_0(\Emb(S^1,U))$ is infinite when $k\geq3$. Nevertheless, since $\GG$ is a collection of pairwise disjoint simple closed geodesics, there are only finitely many homotopy classes $h\in \pi_0(\Emb(S^1,U))$ containing elements of $\GG\setminus\GG'$. 
For every such homotopy class $h$, let $\GG_h$ be the subcollection of those $\gamma\in\GG\setminus\GG'$ contained in $h$. For each pair of distinct $\gamma_1,\gamma_2\in\GG_h$, we denote by $A_{\gamma_1,\gamma_2}\subset B$ the compact annulus with boundary $\partial A_{\gamma_1,\gamma_2}=\gamma_1\cup\gamma_2$. Consider the constant $a_0>0$ provided by Lemma~\ref{l:annulus}.
Since $\area(U,g)<\infty$, there exist finitely many $\gamma_{h,1},...,\gamma_{h,n_h}\in\GG_h$ such that, for each $\gamma\in\GG_h\setminus\{\gamma_{h,1},...,\gamma_{h,n_h}\}$, we have $\area(A_{\gamma,\gamma_{h,i}},g)\leq a_0$ for some $i$. We consider the finite subcollection $\GG''\subset\GG$ that comprises $\GG'$ and all the elements $\gamma_{h,j}$. Summing up, we showed that, for each $\gamma\in\GG\setminus\GG''$, there exists $\zeta\in\GG''$ such that $\area(A_{\gamma,\zeta})\leq a_0$. This, together with Lemma~\ref{l:annulus}, implies
\begin{align*}
\ell
:=
\sup_{\gamma\in\GG} L(\gamma) 
\leq 
\max_{\gamma\in\GG''} \tfrac92 L(\gamma) 
<
\infty.
\end{align*}

Let $\KK\subset W^{1,2}(S^1,M)$ be the subspace of closed geodesics $\gamma:S^1\to M$ of length $L(\gamma)\leq\ell$, pa\-ram\-etrized with constant speed.
As a consequence of the Palais--Smale condition for the geodesic energy functional \cite[Theorem~1.4.7]{Klingenberg:1978wy}, $\KK$ is compact in the $W^{1,2}$ topology.

Assume by contradiction that there exist an infinite sequence $\gamma_n\in\GG\subset\KK$. Up to extracting a subsequence, $\gamma_n$ converges to a simple closed geodesic $\gamma$ in the $W^{1,2}$ topology. A priori, $\gamma$ is contained in the closure $\overline B$, but Lemma~\ref{l:gamma_0} guarantees that $\gamma$ is actually contained in the open ball $B$.
Since the $\gamma_n$'s are pairwise disjoint, they are also disjoint from their limit $\gamma$. This, together with Corollary~\ref{c:intersecting}(ii), implies that $\gamma$ is without conjugate points, and therefore it must be non-degenerate by our assumption. This gives a contradiction, since a non-degenerate closed geodesic $\gamma$ cannot be the $W^{1,2}$-limit of closed geodesics that are disjoint from $\gamma$.
\end{proof}

\begin{proof}[Proof of Theorem~\ref{t:scg2}]

Let $\gamma_1,...,\gamma_n$ be a maximal sequence of pairwise disjoint simple closed geodesics contained in the open ball $B$ such that 
$\gamma_{i+1}\subset B_{\gamma_i}$ for each $i=1,...,n-1$. Lemma~\ref{l:V_finite} implies that such a sequence must be finite. The maximality of this sequence implies that every other simple closed geodesic $\gamma$ contained in $B$ and not intersecting any $\gamma_i$ must have filling $B_\gamma$ contained in the complement $U:=B\setminus(\gamma_1\cup...\cup\gamma_n)$. Notice that the first simple closed geodesic $\gamma_1$ must be the waist provided by Lemma~\ref{l:gamma_0}, whose filling $B_{\gamma_1}$ contains any other closed geodesic contained in $B$. We recall that, by our assumption on $B$, any simple closed geodesic contained in $B$ that is not a waist must have conjugate points.

For each $i\in\{1,...,n-1\}$, $\gamma_i$ or $\gamma_{i+1}$ must be a waist. Indeed, if none of them were waists, Lemma~\ref{l:Bangert} would imply that the open annulus $A_i:=B_{\gamma_{i}}\setminus\overline{B_{\gamma_{i+1}}}$ contains a non-contractible embedded circle $\zeta_0$ with length $L(\zeta_0)<\min\{L(\gamma_{i}),L(\gamma_{i+1})\}$, and Lemma~\ref{l:convergence_scg} would imply that there exists a non-contractible closed geodesic $\zeta\subset\overline {A_i}$ such that $L(\zeta)\leq L(\zeta_0)$; these inequalities would imply that $\zeta$ is contained in $A_i$, contradicting the maximality of the sequence $\gamma_1,...,\gamma_n$.
Moreover, $\gamma_i$ and $\gamma_{i+1}$ cannot both be waists. Otherwise Lemma~\ref{l:minmax}(i) would imply that the open annulus $A_i$ contains a non-contractible simple closed geodesic $\zeta$, contradicting the maximality of the sequence $\gamma_1,...,\gamma_n$.  
Finally, the last simple closed geodesic $\gamma_n$ of the sequence cannot be a waist, for otherwise Lemma~\ref{l:minmax}(ii) would imply that its filling $B_\gamma$ contains another simple closed geodesic, contradicting once more the maximality of the sequence $\gamma_1,...,\gamma_n$. Summing up, we proved that $n$ is even, $\gamma_i$ is a waist if $i$ is odd, whereas $\gamma_i$ has conjugate points if $i$ is even.

We denote by $\KK=\{\gamma_1,\gamma_3,...,\gamma_{n-1}\}$ the collection of waists in the sequence $\gamma_1,...,\gamma_{n}$, and we set
\begin{align*}
 K:=\bigcup_{\gamma\in\KK} (\dot\gamma\cup-\dot\gamma).
\end{align*}
Lemma~\ref{l:Ws_waist} implies 
\[\Ws(K)\setminus K\subseteq\trap_+(SU).\]
Conversely, consider an arbitrary point $z\in\trap_+(SU)$, so that $\phi_t(z)\in SU$ for all $t>0$ large enough. We shall prove that the $\omega$-limit $\omega(z)$ is either $\dot\gamma$ or $-\dot\gamma$ for some $\gamma\in\KK$, which will imply the opposite inclusion
\begin{align*}
\trap_+(SU) \subseteq \Ws(K)\setminus K.
\end{align*}
Let $\pi:SM\to M$ be the base projection. Since $z\in\trap_+(SU)$, there exists a connected component $W\subset U$ such that $\phi_t(z)\in SW$ for all $t>0$ large enough, and therefore $\pi(\omega(z))\subset\overline W$. We have three possible cases.\vspace{3pt}

\emph{Case 1}: $W=B_{\gamma_n}$. Since $\gamma_n$ has conjugate points, Corollary~\ref{c:intersecting}(ii) implies that the geodesic $\pi\circ\phi_t(z)$ does not enter a sufficiently small neighborhood of $\partial B_{\gamma_n}$ for all $t>0$ large enough. Therefore, $\pi(\omega(z))$ is contained in the open disk $B_{\gamma_n}$. Since $\omega(z)$ is compact, so is its projection $\pi(\omega(z))$.
Lemma~\ref{l:weakly_locally_convex} implies that any path-connected component of the open set $B_{\gamma_n}\setminus\pi(\omega(z))$ is weakly convex, and Lemma~\ref{l:Bangert} implies that $B_{\gamma_n}\setminus\pi(\omega(z))$ contains a non-contractible embedded loop $\zeta_0$ with length $L(\zeta_0)<L(\gamma_n)$. We can now apply Lemma~\ref{l:convergence_scg}, which provides a sequence of non-contractible smooth embedded circles in $B_{\gamma_n}\setminus\pi(\omega(z))$ converging in the $C^2$ topology to a simple closed geodesic $\zeta\subset\overline{B_{\gamma_n}\setminus\pi(\omega(z))}$ with length $L(\zeta)\leq L(\zeta_0)<L(\gamma_n)$. These inequalities readily imply that $\zeta$ is contained in the open disk $B_{\gamma_n}$, which contradicts the maximality of the sequence $\gamma_1,...,\gamma_n$. This proves that case 1 cannot occur.\vspace{3pt}

\emph{Case 2}: $W=B\setminus \overline{B_{\gamma_1}}$. Since $B$ satisfies Assumption~\ref{a:convexity}, the geodesic $\pi\circ\phi_t(z)$ does not enter a sufficiently small neighborhood of $\partial B$ for all $t>0$ large enough, for otherwise it would intersect $\partial B$ for $t>0$ arbitrarily large. Therefore the compact set $\pi(\omega(z))$ is contained in $\overline W\setminus\partial B$. Let $V\subset W\setminus\pi(\omega(z))$ be the connected component whose boundary $\partial V$ contains $\partial B$. By Lemma~\ref{l:weakly_locally_convex}, $V$ is weakly convex. Lemma~\ref{l:convergence_scg} provides a sequence of non-contractible smooth embedded loops $\zeta_k\subset V$ converging in the $C^2$ topology to a simple closed geodesic $\zeta\subset \overline V$. Notice that $\zeta$ cannot intersect $\partial B$, for otherwise $\zeta_{k}$ would intersect $\partial B$ as well for $k$ large enough according to Assumption~\ref{a:convexity}. Therefore, Lemma~\ref{l:gamma_0} implies that $\zeta=\gamma_1$, and we have $V=B\setminus \overline{B_{\gamma_1}}$ and $\pi(\omega(z))=\gamma_1$. Since $\omega(z)$ is connected and invariant under the geodesic flow, we infer that either $\omega(z)=\dot\gamma_1$ or $\omega(z)=-\dot\gamma_1$. 
\vspace{3pt}

\emph{Case 3}: $W=B_{\gamma_i}\setminus \overline{B_{\gamma_{i+1}}}$ for some $i\in\{1,...,n-1\}$. Let us assume that $i$ is odd, so that $\gamma_i$ is a waist and $\gamma_{i+1}$ has conjugate points (the case of $i$ even is analogous). Since $\gamma_{i+1}$ has conjugate points, Corollary~\ref{c:intersecting}(ii) implies that there exists a neighborhood of $\gamma_{i+1}$ that does not intersect $\pi(\omega(z))$. If $\pi(\omega(z))$ intersects the interior $W$, then we can apply Lemma~\ref{l:convergence_scg} as in case 1, and infer the existence of a simple closed geodesic $\zeta\subset W$, contradicting the maximality of the sequence $\gamma_1,...,\gamma_n$. This proves that 
$\pi(\omega(z))\subseteq\gamma_i$. 
Since $\omega(z)$ is connected and invariant under the geodesic flow, we infer that either $\omega(z)=\dot\gamma_i$ or $\omega(z)=-\dot\gamma_i$.\vspace{3pt}

Summing up, we proved that
\begin{align*}
\trap_+(SU) = \Ws(K)\setminus K.
\end{align*}
Since $\trap_-(SU)=-\trap_+(SU)$ and $\Wu(K)=-\Ws(K)$, we also have
\begin{align*}
\trap_-(SU) = \Wu(K)\setminus K.
\end{align*}
It remains to show that no complete geodesic is entirely contained in $U$. Assume by contradiction that this does not hold, so that there exists a point $z\in SU$ such that $\phi_t(z)\in SU$ for all $t\in\R$. In particular, $z\in\trap_+(SU)\cap\trap_-(SU)$, and therefore $z\in\Wu(\dot\gamma\cup-\dot\gamma)\cap\Ws(\dot\gamma\cup-\dot\gamma)\setminus(\dot\gamma\cup-\dot\gamma)$ for some waist $\gamma\in\KK$. Let $W\subset U$ be the connected component such that $\phi_t(z)\in SW$ for all $t\in\R$. The waist $\gamma$ is a connected component of $\partial W$, and therefore we must have either $W=B\setminus \overline{B_{\gamma_1}}$ or $W=B_{\gamma_i}\setminus \overline{B_{\gamma_{i+1}}}$ for some $i\in\{1,...,n-1\}$. The compact subset 
\[C:=\overline{\bigcup_{t\in\R}\phi_t(z)}\subseteq\dot\gamma\cup-\dot\gamma\cup\bigcup_{t\in\R}\phi_t(z)\]
is invariant under the geodesic flow, and therefore any path-connected component of $W\setminus\pi(C)$ is weakly convex according to Lemma~\ref{l:weakly_locally_convex}. The end of the argument is entirely analogous to the one in cases 2 and 3 above.
There exists a neighborhood of $\partial W\setminus\gamma$ that does not intersect $\pi(C)$. Since $\pi(C)$ intersects the open set $W$, with a suitable application of Lemma~\ref{l:convergence_scg} we infer the existence of a simple closed geodesic in $W$, contradicting the maximality of the sequence $\gamma_1,...,\gamma_n$.
\end{proof}

\section{Construction of surfaces of section}\label{s:proofs}

\subsection{Fried surgery of Birkhoff annuli}\label{ss:Fried}
The surfaces of section provided by our main theorems will be obtained by performing surgeries \`a la Fried \cite{Fried:1983uj} on a suitable collection of Birkhoff annuli of  closed geodesics. In this subsection, we briefly recall this procedure.

Let $(M,g)$ be an oriented Riemannian surface, and $J$ its associated complex structure, i.e.\ $v,Jv$ is an oriented orthonormal basis of $T_xM$ for all tangent vectors $v\in S_xM$. Let $\gamma:\R/\ell\Z\to M$ be a closed geodesic parametrized with unit speed $\|\dot\gamma\|_g\equiv1$.
The \emph{Birkhoff annulus} of $\gamma$ is the immersed compact annulus in $SM$ given by
\begin{align*}
 A(\dot\gamma):=\big\{ v\in S_{\gamma(t)}M\ \big|\ t\in\R/\ell\Z,\   g(J\dot\gamma(t),v)\geq0 \big\}.
\end{align*}
Its boundary $\partial A(\dot\gamma)=\dot\gamma\cup -\dot\gamma$ is embedded, while its interior $\interior(A(\dot\gamma))$ is immersed and transverse to the geodesic vector field $X$ on $SM$.

Assume now that $\gamma$ is simple. In this case, $A(\dot\gamma)$ is embedded, and therefore is a surface of section for $X$. By considering the opposite orientation on $\gamma$, we obtain a second Birkhoff annulus $A(-\dot\gamma)$ such that 
\[\interior(A(\dot\gamma))\cap\interior(A(-\dot\gamma))=\varnothing,
\qquad\partial A(\dot\gamma)=\partial A(-\dot\gamma)=\dot\gamma\cup -\dot\gamma.\]

\begin{figure}
\includegraphics{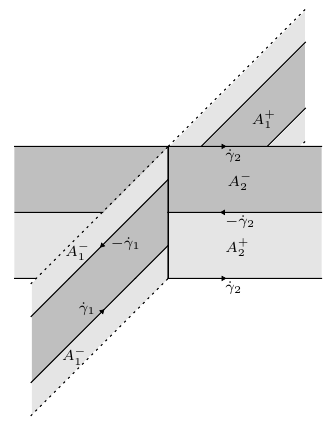}
\caption{The four Birkhoff annuli $A_1^+,A_1^-,A_{2}^+,A_{2}^-$ of the simple closed geodesics $\gamma_1,\gamma_2$. The central vertical line of double points is the fiber $S_xM$ of a point $x\in\gamma_1\cap\gamma_2$. The geodesic vector field should be thought as horizontal, and makes a full turn as we move through the line of double points.}
\label{f:singular_sos}
\end{figure}

For any pair of distinct simple closed geodesics $\gamma_1,\gamma_2$ of $(M,g)$ that intersect each other, the Birkhoff annuli $A_i^\pm:=A(\pm\dot\gamma_i)$ intersect as in Figure~\ref{f:singular_sos}. We see their union 
\[\Upsilon:=A_1^+\cup A_1^-\cup A_2^+\cup A_2^-\] as a non-injectively immersed surface with boundary 
\[\partial\Upsilon
=
\dot\gamma_1 \cup -\dot\gamma_1\cup\dot\gamma_2\cup-\dot\gamma_2.\]  Notice that, for any $\ell>\max\{L(\gamma_1),L(\gamma_2)\}$, there exists an open neighborhood $N\subset\Upsilon$ of $\partial\Upsilon$ such that, for each $z\in N$, the orbit segment $\phi_{(0,\ell]}(z)$ intersects $\Upsilon$.

We apply the following surgery procedures, due to Fried, in order to resolve the self-intersections with interior points of $\Upsilon$, and produce a surface of section $\Sigma\subset SM$ with the same boundary $\partial\Sigma=\partial\Upsilon$. Away from an arbitrarily small neighborhood $U$ of the subspace of self-intersections of $\interior(\Upsilon)$ with $\Upsilon$, we set $\Sigma\cap (SM\setminus U):=\Upsilon\cap (SM\setminus U)$. Along the lines of double points in $\interior(\Upsilon)$, we resolve the self-intersections and obtain $\Sigma$ as in Figure~\ref{f:surgery1}. Finally, near the intersections $\interior(\Upsilon)\cap\partial\Upsilon$, the surface $\Sigma$ is obtained as in Figure~\ref{f:surgery2}. Up to choosing the neighborhood $U$ where the surgery takes place to be small enough, for each $z\in\Upsilon$ the orbit segment $\phi_{(-1,1)}(z)$ intersects $\Sigma$; analogously, for each $z\in\Sigma$, the orbit segment $\phi_{(-1,1)}(z)$ intersects $\Upsilon$.

\begin{figure}
\includegraphics{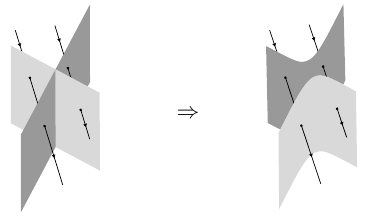}
\caption{Fried surgery to resolve the double points in $\interior(\Upsilon)$.}
\label{f:surgery1}
\end{figure}

\begin{figure}
\includegraphics{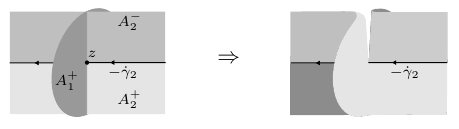}
\caption{Fried surgery to resolve the lines of double points near the boundary point $z\in\partial\Upsilon$.}
\label{f:surgery2}
\end{figure}

If a closed geodesic $\gamma$ is not simple, one can apply Fried surgeries to resolve the self-intersections of its Birkhoff annulus $A(\dot\gamma)$ and produce a surface of section. More generally, one can apply Fried surgeries to produce a surface of section out of the Birkhoff annuli of any finite collection of closed geodesics.

\subsection{Complete system of closed geodesics}\label{ss:complete_system}
Let $(M,g)$ be a closed connected oriented Riemannian surface with unit tangent bundle $\pi:SM\to M$ and geodesic flow $\phi_t:SM\to SM$. We introduce the following notion, which we will employ in the next subsection.

\begin{Definition}\label{d:complete_system}
 A \emph{complete system of closed geodesics} $\GG$ is a non-empty finite collection of  closed geodesics of $(M,g)$, together with a (possibly empty) subcollection $\KK\subset\GG$, satisfying the following properties. We set
\begin{align*}
 K:=\bigcup_{\gamma\in\KK} (\dot\gamma\cup-\dot\gamma),
 \qquad
 R:=\bigcup_{\gamma\in\GG\setminus\KK} (\dot\gamma\cup-\dot\gamma),
 \qquad
 U:=M\setminus\bigcup_{\gamma\in\GG}\gamma.
\end{align*}
\begin{itemize}
\item[(i)] Every $\gamma\in\KK$ is a non-degenerate contractible waist disjoint from all the other closed geodesics in $\GG\setminus\{\gamma\}$. In particular,  $K$ is a hyperbolic invariant subset for the geodesic flow.\vspace{2pt}

\item[(ii)] No complete orbit $\phi_{(-\infty,\infty)}(z)$ is entirely contained in $SU$, and the trapped sets of $SU$ are given by
\begin{align*}
 \trap_+(SU)=\Ws(K)\setminus K,
 \qquad
 \trap_-(SU)=\Wu(K)\setminus K.
\end{align*}
For this reason, we briefly call $\KK$ the \emph{limit subcollection} of $\GG$, and $K$ the \emph{limit set}.\vspace{2pt}

\item[(iii)] The invariant subset $R$ admits an open neighborhood $N\subset SM$ and a positive number $\ell>0$ such that, for each $z\in N$, the orbit segment $\phi_{(0,\ell]}(z)$ is not contained in $SU$. \hfill\qed
\end{itemize}
\end{Definition}

We stress that the closed geodesics in $\GG\setminus\KK$ are not necessarily simple, and can have mutual intersections as well.
A complete system of closed geodesics with empty limit subcollection will produce a Birkhoff section. For this purpose, we will need the following lemma.

\begin{Lemma}
\label{l:empty_limit_subcollection}
If the limit subcollection $\KK$ is empty, there is a constant $\ell'>0$ such that every geodesic segment of length $\ell'$ intersects some geodesic in $\GG$.
\end{Lemma}

\begin{proof}
We denote by $\Upsilon$ the union of the Birkhoff annuli of the closed geodesics in $\GG$, i.e.
\begin{align*}
 \Upsilon := \bigcup_{\gamma\in\GG} \big( A(\dot\gamma)\cup A(-\dot\gamma) \big).
\end{align*}
Since $\KK=\varnothing$, we have empty trapped sets $\trap_\pm(SU)=\varnothing$. Therefore, for each $z\in SM\setminus\Upsilon$, there exists $\ell_z>1$ such that the orbit segment $\phi_{[1,\ell_z-1]}(z)$ intersects $\interior(\Upsilon)$ transversely. The transversality guarantees that there exists an open neighborhood $N_z\subset SM$ of $z$ such that, for each $z'\in N_z$, the orbit segment $\phi_{[0,\ell_z]}(z')$ intersects $\interior(\Upsilon)$. Consider now the constant $\ell>0$ and the open neighborhood $N$ of $R$ given by property (iii) above. Since $SM$ is compact, there exists finitely many points $z_1,...,z_n\in SM$ such that $N\cup N_{z_1}\cup...\cup N_{z_n}=SM$. For $\ell':=\max\{\ell,\ell_{z_1},...,\ell_{z_n}\}$, we conclude that, for each $z\in SM$, the orbit segment $\phi_{[0,\ell']}(z)$ intersects $\Upsilon$. Namely, every geodesic segment of length $\ell'$ intersects some geodesic in $\GG$.
\end{proof}

The arguments of Contreras--Mazzucchelli's \cite[Section~4]{Contreras:2021vx} allow us to produce, out of a suitable complete system of closed geodesics with non-empty limit subcollection, a new such complete system with strictly smaller limit subcollection. The setting of \cite[Section~4]{Contreras:2021vx} employs Colin--Dehornoy--Rechtman's broken book decompositions \cite{Colin:2020tl}, but it turns out that the arguments go through in our simpler setting as well. We include the details in the rest of this section, for the reader's convenience. We begin with the following preliminary lemma due to Colin--Dehornoy--Rechtman \cite[Lemma~4.9]{Colin:2020tl}, which is based on an argument originally due to Hofer--Wysocki--Zehnder \cite[Proposition~7.5]{Hofer:2003wf}.

\begin{Lemma}
\label{l:homoclinics}
If the limit set $K$ is non-empty and satisfies the transversality condition 
\[\Wu(K)\pitchfork\Ws(K)
,\] 
then there exists $\gamma\in\KK$ such that $W\cap\Ws(\dot\gamma)\neq\varnothing$ for each path-connected component $W\subseteq\Wu(\dot\gamma)\setminus\dot\gamma$. Namely, the closed orbit $\dot\gamma$ has homoclinics in all path-connected components of $\Wu(\dot\gamma)\setminus\dot\gamma$.
\end{Lemma}

\begin{proof}
We first show that, for each $\gamma\in\KK$ and for each connected component $W\subseteq\Wu(\dot\gamma)\setminus\dot\gamma$, there exists a heteroclinic  
\begin{align}
\label{e:heteroclinic}
W\cap\Ws(K)\neq\varnothing 
\end{align}
Let us prove this by contradiction, assuming that $W\cap\Ws(K)=\varnothing$. 

The path-connected component $W$ is an immersed cylinder in $SM$ with one end equal to $\dot\gamma$. Let $S\subset W$ be an embedded essential circle that is $C^1$-close to $\dot\gamma$ and transverse to the geodesic vector field, so that its base projection $\pi(S)$ does not intersect $\gamma$ (Lemma~\ref{l:Ws_waist}), nor any other closed geodesic in the collection $\GG$. We consider the union of the Birkhoff annuli
\begin{align*}
\Upsilon':=\bigcup_{\zeta\in\GG\setminus\{\gamma\}} \big( A(\dot\zeta)\cup A(-\dot\zeta) \big),
\end{align*}
and an open neighborhood $V\subset SM$ of $\Upsilon'$ such that $V\cap \phi_{(-\infty,0]}(S)=\varnothing$. We apply Fried surgery as explained in Section~\ref{ss:Fried} in order to resolve the self-intersection points of $\Upsilon'$ with its interior (if there are any), and produce  surfaces of section $\Sigma'\subset V$ such that 
\begin{align}
\label{e:Fried_correspondence}
\phi_{(-1,1)}(z)\cap\Sigma'\neq\varnothing,\ \ \forall z\in\Upsilon'.
\end{align}
We set 
\[\Upsilon=\Upsilon'\cup A(\dot\gamma)\cup A(-\dot\gamma),\qquad \Sigma=\Sigma'\cup A(\dot\gamma)\cup A(-\dot\gamma).\]
For each $z\in S$, we have $\phi_{(-\infty,0]}(z)\subset SU$, whereas $\phi_{(-\infty,\infty)}(z)$ is not entirely contained in $SU$, and thus intersects $\Upsilon$. This, together with~\eqref{e:Fried_correspondence}, implies that there exists a minimal $t_z>0$ such that the orbit $t\mapsto\phi_{t}(z)$ intersects $\Sigma$ transversely for $t=t_z$. This transversality, together with the compactness of the circle $S$, implies that the function $z\mapsto t_z$ is smooth on $S$, and we obtain an embedded circle
\begin{align*}
 S_0:=\big\{\phi_{t_z}(z)\ \big|\ z\in S\big\}\subset W\cap\Sigma.
\end{align*}
We consider the first return time
\begin{align*}
\tau:\interior(\Sigma)\to(0,+\infty],
\qquad
\tau(z):=\inf\big\{ t>0\  \big|\ \phi_{t}(z)\in\Sigma \big\},
\end{align*}
which is smooth on the open subset $\Sigma_0:=\tau^{-1}(0,\infty)$. The first return map
\begin{align*}
 \psi:\Sigma_0\to\interior(\Sigma),\qquad\psi(z)=\phi_{\tau(z)}(z)
\end{align*}
is a diffeomorphism onto its image that preserves the area form $d\lambda|_{\interior(\Sigma)}$, where $\lambda$ is the Liouville contact form of $SM$. 

Since $S_0\cap\Ws(K)=\varnothing$, we have $S_0\cap\trap_+(SU)=\varnothing$, and therefore $\psi^n(S_0)\subset\Sigma_0$ for all $n\geq0$. We obtained an infinite sequence  $S_n:=\psi^n(S_0)$ of pairwise disjoint embedded circles in the interior of the surface of section $\Sigma$. Since the unstable manifold $\Wu(\dot\gamma)$ is an immersed surface tangent to the geodesic vector field, the 2-form $d\lambda|_W$ vanishes identically. If $S_n$ bounds a disk $B_n\subset\interior(\Sigma)$, we denote by $A_n\subset W$ the compact annulus with boundary $\partial A_n=\dot\gamma\cup S_n$, and Stokes' theorem implies
\begin{align*}
\area(B_n,d\lambda)=\int_{B_n} d\lambda=\int_{A_n\cup B_n} d\lambda = \int_{\dot\gamma} \lambda = L(\gamma),
\end{align*}
where $L(\gamma)>0$ is the length of the simple closed geodesic $\gamma$. In particular, if $S_{n_1},S_{n_2}$ bound disks $B_{n_1},B_{n_2}\subset\interior(\Sigma)$ for some distinct $n_1,n_2\geq0$, we have $B_{n_1}\cap B_{n_2}=\varnothing$.
This, together with the finiteness of the area
\begin{align*}
\area(\interior(\Sigma),d\lambda)
=
\int_{\Sigma} d\lambda = \int_{\partial\Sigma}\lambda,
\end{align*}
readily implies that there exist at most finitely many $n\geq0$ such that $S_n$ is contractible in $\interior(\Sigma)$. Therefore, since the embedded circles $S_n$ are pairwise disjoint, there exist distinct $n_1,n_2\geq0$ and an embedded compact annulus $A\subset\interior(\Sigma)$ with boundary $\partial A=S_{n_1}\cup S_{n_2}$. Let $A'\subset W$ be the embedded compact annulus with boundary $\partial A'=S_{n_1}\cup S_{n_2}$, so that the union $A\cup A'$ is a piecewise smooth embedded torus in $SM$. By Stokes theorem, we have
\begin{align*}
 \int_A d\lambda = \int_{A\cup A'} d\lambda = 0,
\end{align*}
which contradicts the fact that $d\lambda|_{A}$ is an area form. This proves the existence of heteroclinics~\eqref{e:heteroclinic}.

Our transversality assumption on the heteroclinics, together with the shadowing lemma from hyperbolic dynamics \cite[Th.~5.3.3]{Fisher:2019vz}, implies that, for all $\gamma_1,\gamma_2,\gamma_3\in\KK$ and for each path-connected component $W\subset\Wu(\dot\gamma_1)\setminus\dot\gamma_1$, if there are heteroclinics $W\cap\Ws(\dot\gamma_2)\neq\varnothing$ and  $\Wu(\dot\gamma_2)\cap\Ws(\dot\gamma_3)\neq\varnothing$, then there are also heteroclinics $W\cap\Ws(\dot\gamma_3)\neq\varnothing$.

For each $\gamma\in\KK$, we fix arbitrary path-connected components $W_{\pm\dot\gamma}\subset\Wu(\pm\dot\gamma)\setminus\pm\dot\gamma$. We already proved that every such path-connected component must contain a heteroclinic to $K$. Therefore, there exists a sequence of oriented waists $\gamma_i\in\KK$ with heteroclinics $W_{\dot\gamma_i}\cap\Ws(\dot\gamma_{i+1})\neq\varnothing$. The same waist may appear different times with opposite orientations in the sequence. Nevertheless, since the collection $\GG$ is finite, there exists $n\leq 2\,\#\GG+1$ such that $\gamma_1=\gamma_n$ as oriented waists. This implies that $W_{\dot\gamma_1}\cap\Ws(\dot\gamma_1)\neq\varnothing$.
\end{proof}

\begin{Lemma}
\label{l:reduction}
If the limit set $K$ is non-empty and satisfies the transversality condition 
\[\Wu(K)\pitchfork\Ws(K),\] 
then there exists $\gamma\in\KK$ and  another complete system of closed geodesics with limit subcollection contained in $\KK\setminus\{\gamma\}$.
\end{Lemma}

\begin{proof}
By Lemma~\ref{l:homoclinics}, there exists $\gamma\in\KK$ whose associated periodic orbit $\dot\gamma$ has homoclinics in all path-connected components of $\Wu(\dot\gamma)\setminus\dot\gamma$. Since $\gamma$ is a non-degenerate waist, Lemma~\ref{l:Ws_waist} implies that there exists a tubular neighborhood $A\subset M$ of $\gamma$ and an open neighborhood $V\subset SM$ of $\dot\gamma$ such that, if we denote by $W\subset V\cap\Wu(\dot\gamma)$ the path-connected component containing $\dot\gamma$, the restriction of the base projection  $\pi|_W:W\to A$ is a diffeomorphism. An analogous statement holds for the stable manifold $\Ws(\dot\gamma)$. Therefore, for any homoclinic point $z\in\Wu(\dot\gamma)\cap\Ws(\dot\gamma)\setminus\dot\gamma$, the corresponding geodesic $\zeta(t):=\pi\circ\phi_t(z)$ is contained in $A\setminus\gamma$ provided $|t|$ is large enough.

The complement $A\setminus\gamma$ is the disjoint union of two open annuli $A_1$ and $A_2$, and therefore $W\setminus\dot\gamma$ is the disjoint union of the open annuli $W_1:=\pi|_W^{-1}(A_1)$ and $W_2:=\pi|_W^{-1}(A_2)$. Our assumption on $\gamma$ implies that both $W_1$ and $W_2$ intersect the stable manifold $\Ws(\dot\gamma)$, and we fix homoclinic points $z_i\in W_i\cap\Ws(\dot\gamma)$ such that, if we denote by $\zeta_i(t):=\pi\circ\phi_t(z)$ the associated geodesics, we have $\zeta_i(t)\in A_i$ for all $t\leq0$. We have two possible cases.\vspace{2pt}

\emph{Case 1}: There exists $i\in\{1,2\}$ such that $\zeta_i(t)\in A_{3-i}$ for all $t>0$ large enough. For each $\delta>0$, there exist arbitrarily large $a,b>0$ such that the points $\phi_{-a}(z_i)$ and $\phi_{b}(z_i)$ are $\delta$-close (where the distance on $SM$ is the one induced by the Riemannian metric $g$). We consider the orbit segment
\begin{align*}
\Gamma_\delta:[-a,b]\to SM,\qquad \Gamma_\delta(t)=\phi_{t}(z_i),
\end{align*}
and we extend it as a discontinuous periodic curve of period $\tau:=a+b$.\vspace{2pt}

\emph{Case 2}: For all $i\in\{1,2\}$, we have $\zeta_i(t)\in A_{i}$ for all $t>0$ large enough. For each $\delta>0$, there exist arbitrarily large $a_1,b_1,a_2,b_2>0$ such that the points $\phi_{-a_i}(z_i)$ and $\phi_{b_i}(z_i)$ are $\delta$-close. We consider the orbit segments 
\[\Gamma_{\delta,i}:[-a_i,b_i]\to SM,\qquad \Gamma_{\delta,i}(t)=\phi_{t}(z_i),\]
and we define $\Gamma_\delta$ to be the discontinuous periodic curve of period $\tau:=a_1+b_1+a_2+b_2$ obtained by extending periodically the concatenation of $\Gamma_{\delta,1}$ and $\Gamma_{\delta,2}$.\vspace{2pt}
 
In both cases, for each $\delta>0$ we obtained a periodic $\delta$-pseudo orbit $\Gamma_\delta$ of the geodesic flow with arbitrarily large minimal period. Moreover, the projection $\pi\circ\Gamma_\delta$ makes a $\delta$-jump from $A_1$ to $A_2$. The shadowing lemma from hyperbolic dynamics \cite[Th.~5.3.3]{Fisher:2019vz} implies that, for each $\epsilon>0$, there exist $\delta>0$ and $z\in SM$ such that the orbit $t\mapsto\phi_t(z)$ is periodic and pointwise $\epsilon$-close to $\Gamma_\delta$ up to a time reparametrization with Lipschitz constant $\epsilon$. Up to choosing $\epsilon>0$ small enough, the corresponding closed geodesic $\zeta(t):=\pi\circ\phi_t(z)$ must intersect $\gamma$ transversely.

Let $\KK'$ be the subcollection of those simple closed geodesics $\eta\in\KK$ that intersect $\zeta$, which is non-empty since it contains $\gamma$. The collection $\GG\cup\{\zeta\}$ is a complete system of  closed geodesics with limit subcollection $\KK\setminus\KK'$.
\end{proof}

\subsection{Proof of Theorems~\ref{mt:Birkhoff},~\ref{mt:Bangert},~\ref{mt:Kupka_Smale},~\ref{mt:broken_book}}

We first single out a suitable family of non-contractible waists, which is always available in every closed orientable Riemannian surface of positive genus.

\begin{Lemma}
\label{l:simple_closed_geodesics}
On any closed orientable Riemannian surface $(M,g)$ of genus $G\geq1$, there exist waists $\gamma_1,...,\gamma_{2G}$ such that:
\begin{itemize}
\item[(i)] $\gamma_i\cap\gamma_{i+1}$ is a singleton for all $i\in\{1,...,2G-1\}$,
\item[(ii)] $\gamma_i\cap \gamma_j=\varnothing$ if $|i-j|\geq2$,
\item[(iii)] $M\setminus(\gamma_1\cup...\cup\gamma_{2G})$ is simply connected,
\item[(iv)] every $\gamma_i$ is a waist.
\end{itemize}
\end{Lemma}

\begin{proof}
We first consider non-contractible embedded loops $\zeta_1,...,\zeta_{2G}\subset M$ such that $\zeta_i\cap\zeta_{i+1}$ is a singleton for all $i\in\{1,...,2G-1\}$, $\zeta_i\cap \zeta_j=\varnothing$ if $|i-j|\geq2$, and $M\setminus(\zeta_1\cup...\cup\zeta_{2G})$ is simply connected (Figure~\ref{f:genus_G}). 
We denote by $\CC_i\subset\Emb(S^1,M)$ the connected component containing $\zeta_i$. Notice that the $\CC_i$'s are pairwise distinct, and the embedded loops $\zeta_1,...,\zeta_{2G}$ are in minimal position, i.e.
\[
\#(\eta_i\cap\eta_j) \geq \#(\zeta_i\cap\zeta_j),\qquad\forall i<j,\ \eta_i\in\CC_i,\ \eta_j\in\CC_j.
\]
By Lemma~\ref{l:convergence_scg}, for each $i=1,...,2G$, there exists a waist $\gamma_i\in\CC_i$ such that
\begin{align}
\label{e:min_scg}
L(\gamma_i)=\min_{\gamma\in\CC_i} L(\gamma).
\end{align}

\begin{figure}
\includegraphics{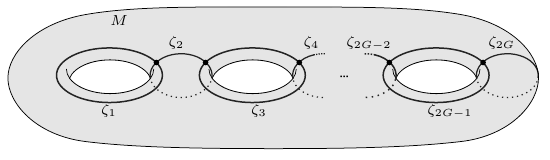}
\caption{The embedded loops $\zeta_1,...,\zeta_{2G}$.}
\label{f:genus_G}
\end{figure}

We recall that a geodesic bigon is a compact disk $D\subset M$ whose boundary is the union of two geodesic arcs. For each $i<j$, there is no geodesic bigon bounded by an arc in $\alpha_i\subset\gamma_i$ and an arc $\alpha_j\subset\gamma_j$. Indeed, if there were such a geodesic bigon with $L(\alpha_i)\leq L(\alpha_j)$, we could further shrink the simple closed geodesics $\gamma_j$, and obtain a shorter embedded loop in the same homotopy class $\CC_j$, contradicting~\eqref{e:min_scg}. By the bigon criterium \cite[Prop.~1.7]{Farb:2012ws}, the waists $\gamma_1,...,\gamma_{2G}$ are in minimal position. Therefore 
\[
\#(\gamma_i\cap\gamma_j) = \#(\zeta_i\cap\zeta_j),\qquad\forall i<j.
\]

In order to prove that $B:=M\setminus(\gamma_1\cup...\cup\gamma_{2G})$ is simply connected, we cut open the surface $M$ at the simple closed geodesics $\gamma_2,\gamma_4,...,\gamma_{2G}$, thus obtaining a compact surface $M'$ of genus zero with $2G$ boundary components. Each waist $\gamma_{2i}$, with $i=1,...,G$, corresponds to two boundary components $\gamma_{2i}^+,\gamma_{2i}^-\subset\partial M'$. Each waist $\gamma_{2i+1}$, with $i=1,...,G-1$, splits into two embedded arcs $\gamma_{2i+1}^+,\gamma_{2i+1}^-$; the arc $\gamma_{2i+1}^\pm$ joins the boundary components $\gamma_{2i}^\pm$ and $\gamma_{2i+2}^\pm$. Finally, the waist $\gamma_1$ corresponds to an embedded arc $\gamma_1^+$ joining the boundary components $\gamma_2^+$ and $\gamma_2^-$. We now cut open $M'$ along the embedded arcs $\gamma_{2i+1}^\pm$, obtaining a compact disk $M''$. The interior $\interior(M'')$ is diffeomorphic to $B$.
\end{proof}

By means of Fried surgery, we produce a connected surface of section out of the union of the Birkhoff annuli of the waists given by Lemma~\ref{l:simple_closed_geodesics}.

\begin{Lemma}
\label{l:Sigma}
Let $\gamma_1,...,\gamma_{2G}$ be the waists provided by Lemma~\ref{l:simple_closed_geodesics}, and consider the union of the associated Birkhoff annuli
\begin{align}
\label{e:Upsilon}
 \Upsilon:=\bigcup_{i=1}^{2G}   \big(A(\dot\gamma_i)\cup A(-\dot\gamma_i)\big).
\end{align}
For each open neighborhood $V\subset SM$ of $\Upsilon$, there exists a surface of section \[\iota:\Sigma\looparrowright V\] with the following properties:
\begin{itemize}

\item[(i)] $\Sigma$ is a compact connected surface of genus one and $8G-4$ boundary components.\vspace{2pt}

\item[(ii)] The restriction 
\[\iota|_{\partial\Sigma}:\partial\Sigma\to\bigcup_{i=1}^{2G}(\dot\gamma_i\cup-\dot\gamma_i)\]
is a double cover. The preimages $\iota^{-1}(\dot\gamma_i)$ and $\iota^{-1}(-\dot\gamma_i)$ are connected if $i\in\{1,2G\}$, and have two connected components each if $G>1$ and $2\leq i\leq 2G-1$.\vspace{2pt}

\item[(iii)] There exist $\ell>0$ such that, for each $z\in SM$ sufficiently close to $\partial\Sigma$, we have $\phi_{[-\ell,0)}(z)\cap\Sigma\neq\varnothing$ and $\phi_{(0,\ell]}(z)\cap\Sigma\neq\varnothing$.\vspace{2pt}

\item[(iv)] For each $z\in\Upsilon$ we have $\phi_{(-1,1)}(z)\cap\Sigma\neq\varnothing$, and for each $z'\in\Sigma$ we have $\phi_{(-1,1)}(z')\cap\Upsilon\neq\varnothing$.

\end{itemize}
\end{Lemma}

\begin{proof}
We see $\Upsilon$ as a non-injectively immersed surface in $SM$ with boundary and interior
\[\partial\Upsilon=\bigcup_{i=1}^{2G} \big(\dot\gamma_i\cup-\dot\gamma_i\big),
\qquad
\interior(\Upsilon)=\bigcup_{i=1}^{2G}\big(\interior(A_i^+)\cup\interior(A_i^-)\big),
\] 
where $A_i^\pm:=A(\pm\dot\gamma_i)$. We set
 $\ell_0:=\max\{L(\gamma_1),...,L(\gamma_{2G})\}$. Since $\gamma_i$ intersects $\gamma_{i-1}$ and $\gamma_{i+1}$ for all $2\leq i\leq2G-1$, we have $\phi_{[2,\ell_0+2]}(z)\cap\interior(\Upsilon)\neq\varnothing$ for all $z\in\partial\Upsilon$. Therefore, there exists an open neighborhood $N\subset SM$ of $\partial\Upsilon$ such that
\begin{align}
\label{e:bounded_return_near_boundary_Upsilon}
\phi_{[1,\ell_0+3]}(z)\cap\interior(\Upsilon)\neq\varnothing,\qquad\forall z\in N.
\end{align}

Let $V\subset SM$ be an open neighborhood of $\Upsilon$. We apply Fried surgeries in an arbitrarily small neighborhood of the self-intersection points of $\Upsilon$, as explained in Section~\ref{ss:Fried}, and obtain a surface of section $\iota:\Sigma\looparrowright V$ with the same boundary $\partial\Sigma=\partial\Upsilon$ and satisfying point~(iv) in the statement of the lemma. This, together with~\eqref{e:bounded_return_near_boundary_Upsilon}, implies point (iii) for $\ell:=\ell_0+4$.

We claim that the interior $\interior(\Sigma)$ is path-connected. Indeed near the point in $\interior(A_1^+)\cap \partial A_2^+\cap \partial A_2^-$, Fried surgery glues together the intersecting annuli $A_1^+$, $A_2^-$, and $A_2^+$ in the same connected component of $\interior(\Sigma)$, as is clear from Figures~\ref{f:singular_sos} and~\ref{f:surgery2}; analogously, near the point in $\interior(A_1^-)\cap \partial A_2^+\cap \partial A_2^-$, Fried surgery glues together the intersecting annuli $A_1^-$, $A_2^-$, and $A_2^+$.
 Therefore, the surgery sends the four annuli $A_1^+$, $A_1^-$, $A_2^+$, and $A_2^-$ to the same path-connected component of $\interior(\Sigma)$. Assume by induction that Fried surgery sends the annuli $A_1^+,A_1^-,...,A_{i}^+,A_{i}^-$ to the same path-connected component $W\subset\interior(\Sigma)$. Near the point in $A_i^+\cap \partial A_{i+1}^+\cap \partial A_{i+1}^-$, Fried surgery glues together the intersecting annuli $A_{i}^+$, $A_{i+1}^-$, and $A_{i+1}^+$. Therefore, the surgery sends $A_{i+1}^-$ and $A_{i+1}^+$ to the connected component $W$ as well, and we conclude that $W=\interior(\Sigma)$.

\begin{figure}
\includegraphics{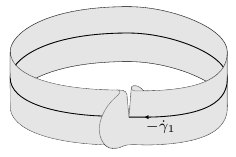}
\caption{The surface $\Sigma$ near the boundary component $-\dot\gamma_1$.}
\label{f:boundary1}
\end{figure}

\begin{figure}
\includegraphics{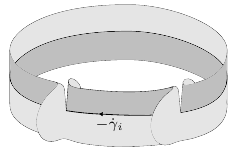}
\caption{The surface $\Sigma$ near the boundary component $-\dot\gamma_i$, for $2\leq i\leq2G-1$. Notice that there are two open annuli in $\interior(\Sigma)$ having boundary on $-\dot\gamma_i$.}
\label{f:boundary2}
\end{figure}

Let us now determine the number of connected components of the boundary $\partial\Sigma$, and how they cover the closed orbits $\pm\dot\gamma_i$. The simple closed geodesic $\gamma_1$ intersects $\gamma_2$ in one point, but does not intersect any of the other simple closed geodesics $\gamma_3,...,\gamma_{2G}$; therefore $\dot\gamma_1$ intersects the interior of the Birkhoff annulus $\interior(A_2^+)$ in one point, but does not intersect any of the other Birkhoff annuli $A_i^\pm$. Analogously, $-\dot\gamma_1$ intersects $\interior(A_2^-)$ in one point but none of the other $A_i^\pm$. Therefore, there is a unique connected component of $\partial\Sigma$ that covers $\dot\gamma_1$, and such connected component winds around $\dot\gamma_1$ twice; analogously, there is a unique connected component of $\partial\Sigma$ that covers $-\dot\gamma_1$, and such connected component winds around $-\dot\gamma_1$ twice (Figure~\ref{f:boundary1}). 
The same conclusions hold for $\gamma_{2G}$.
For each $i\in\{2,...,2G-1\}$,  $\gamma_i$ intersects $\gamma_{i-1}$ and $\gamma_{i+1}$ in one point each, but does not intersect any other $\gamma_j$; therefore, $\dot\gamma_i$ intersects the interiors of the Birkhoff annuli $\interior(A_{i-1}^-)$ and $\interior(A_{i+1}^+)$ in one point each, but does not intersect any of the other Birkhoff annuli $A_j^\pm$; analogously, $-\dot\gamma_i$ intersects $A_{i-1}^+$ and $A_{i+1}^-$ in one point each, but does not intersect any other $A_j^\pm$. Therefore, there are two connected components of $\partial\Sigma$ that are mapped diffeomorphically to $\dot\gamma_i$, and two other connected components of $\partial\Sigma$ that are mapped diffeomorphically to $-\dot\gamma_i$ (Figure~\ref{f:boundary2}). All together, $\partial\Sigma$ has $8G-4$ boundary components.

\begin{figure}
\includegraphics{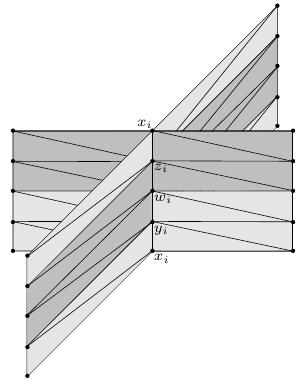}
\caption{Triangulation of $A_i^+\cup A_i^-\cup A_{i+1}^+\cup A_{i+1}^-$.}
\label{f:triangulation}
\end{figure}

Since $\interior(\Sigma)$ is transverse to the geodesic vector field $X$, $\Sigma$ is orientable. Indeed, the Liouville contact form $\lambda$ of $SM$ induces an area form $d\lambda|_{\interior(\Sigma)}$.
In order to compute the genus of $\Sigma$, we first triangulate every Birkhoff annulus $A_i^\pm$ with 6 vertices, 14 edges, and 8 faces. We choose such triangulations so that, along the intersections $A_i^+\cap A_i^-$, $A_i^\pm\cap A_{i+1}^\pm$, and $A_i^\pm\cap A_{i+1}^\mp$, the vertices and edges of $A_{i}^\pm$ and $A_{i+1}^\pm$ match as in Figure~\ref{f:triangulation}. All together, $\Upsilon$ is triangulated with $24\,G$ vertices, $56\,G$ edges, and $32\,G$ faces (once again, we stress that $\Upsilon$ must be seen as an immersed compact surface with boundary, and therefore distinct Birkhoff annuli do not share vertices, edges, or faces). The same number of edges and faces can be used to triangulate $\Sigma$ as well; however, the triangulation of $\Upsilon$ provides one extra vertex for each of the points $x_i,y_i,w_i,z_i$ depicted in Figure~\ref{f:triangulation}, where $i=1,...,2G-1$, and therefore we need to throw away $8\,G-4$ vertices. All together, $\Sigma$ is triangulated with $16\,G+4$ vertices, $56\,G$ edges, and $32\,G$ faces, and therefore has Euler characteristic
\begin{align*}
 \chi(\Sigma)=16\,G+4 - 56\,G + 32\,G = -8\,G + 4.
\end{align*}
Since 
\begin{align*}
-8\,G+4=\chi(\Sigma)=2-2\,\mathrm{genus}(\Sigma) - \#\pi_0(\partial\Sigma)=2-2\,\mathrm{genus}(\Sigma) - 8\,G + 4,
\end{align*}
we conclude that $\Sigma$ has genus one.
\end{proof}

In the following, we shall  employ the notion of complete system of closed geodesics, which we introduced in Definition~\ref{d:complete_system}.
Theorem~\ref{mt:Birkhoff} is a consequence of the following statement.

\begin{Thm}
\label{t:Birkhoff}
Let $(M,g)$ be a closed oriented Riemannian surface of positive genus,  $\gamma_1,...,\gamma_{2G}$ the waists provided by Lemma~\ref{l:simple_closed_geodesics}, and $\Sigma$ the surface of section provided by Lemma~\ref{l:Sigma}. If the open disk $M\setminus(\gamma_1\cup...\cup\gamma_{2G})$ does not contain any simple closed geodesic without conjugate points, then $\Sigma$ is a Birkhoff section.
\end{Thm}

\begin{proof}
With the terminology of Example~\ref{ex:polygon}, the open disk $B:=M\setminus(\gamma_1\cup...\cup\gamma_{2G})$ is a convex geodesic polygon, and in particular satisfies Assumption~\ref{a:convexity}. Since $B$ does not contain any simple closed geodesic without conjugate points, Theorem~\ref{t:scg1} implies that we have empty trapped sets $\trap_+(SB)=\trap_-(SB)=\varnothing$. Therefore, the collection of waists $\{\gamma_1,...,\gamma_{2G}\}$ is a complete system of closed geodesics with empty limit subcollection.
Let $\Upsilon$ be the union of the Birkhoff annuli of the waists $\gamma_1,...,\gamma_{2G}$, as in~\eqref{e:Upsilon}. By Lemma~\ref{l:empty_limit_subcollection}, there exists $\ell>0$ such that, for each $z\in SM$, the orbit segment $\phi_{[0,\ell]}(z)$ intersects $\Upsilon$. This, together with Lemma~\ref{l:Sigma}(iv), implies that for each $z\in SM$ the orbit segment $\phi_{(0,\ell+2)}(z)$ intersects $\Sigma$. Therefore, $\Sigma$ is a Birkhoff section.
\end{proof}

The proof of Theorem~\ref{mt:Bangert} is analogous. We rewrite the statement for the reader's convenience.\vspace{7pt}

\noindent\textbf{Theorem \ref{mt:Bangert}.}
\emph{Let $(S^2,g)$ be a Riemannian 2-sphere, and $\gamma$ a simple closed geodesic with conjugate points whose complement $S^2\setminus\gamma$ does not contain simple closed geodesics without conjugate points. Then both Birkhoff annuli $A(\dot\gamma)$ and $A(-\dot\gamma)$ are Birkhoff sections.} 

\begin{proof}
Both connected components $B_1$ and $B_2$ of the complement $S^2\setminus\gamma$ satisfy the convexity Assumption~\ref{a:convexity}. Since $B_1$ and $B_2$ do not contain any simple closed geodesic without conjugate points, Theorem~\ref{t:scg1} implies that we have empty trapped sets 
$\trap_\pm(S(B_1\cup B_2))=\varnothing$.
Therefore, $\{\gamma\}$ is a complete system of closed geodesics with empty limit subcollection. By Lemma~\ref{l:empty_limit_subcollection}, there exists $\ell>0$ such that any geodesic segment of length $\ell$ intersects $\gamma$. This implies that both Birkhoff annuli $A(\dot\gamma)$ and $A(-\dot\gamma)$ are Birkhoff sections.
\end{proof}

We recall that a consequence of Lusternik-Schnirelmann theorem is that every Riemannian 2-sphere admits a simple closed geodesic with conjugate points (see Remark~\ref{r:S2}).

\begin{Lemma}
\label{l:existence_complete_system}
Let $(M,g)$ be a closed connected oriented Riemannian surface all of whose contractible simple closed geodesics without conjugate points are non-degenerate. If $M\neq S^2$, we consider the collection of waists $\GG'=\{\gamma_1,...,\gamma_{2G}\}$ given by Lemma~\ref{l:simple_closed_geodesics}; if instead $M=S^2$, we set $\GG'=\{\gamma\}$, where $\gamma$ is any simple closed geodesic with conjugate points. There exists a possibly empty finite collection $\GG''$ of contractible simple closed geodesics that are pairwise disjoint and disjoint from all the closed geodesics in $\GG'$, such that $\GG:=\GG'\cup\GG''$ is a complete system of  closed geodesics whose limit subcollection $\KK$ is given by the waists in $\GG''$.
\end{Lemma}

\begin{proof}
 Assume first that $M$ has genus $G\geq1$, so that $\GG'=\{\gamma_1,...,\gamma_{2G}\}$. The complement $B\setminus(\gamma_1\cup...\cup\gamma_{2G})$ satisfies the convexity Assumption~\ref{a:convexity}. 
Since every closed geodesic $\gamma\in\GG'$ intersects some other closed geodesic in $\GG'$, if we fix $\ell> L(\gamma)$, for every $z\in SM$ sufficiently close to $\dot\gamma\cup-\dot\gamma$, the orbit segment $\phi_{(0,\ell]}(z)$ is not contained in $SB$. If $B$ does not contain simple closed geodesics, Theorem~\ref{t:scg1} implies $\trap_\pm(SB)=\varnothing$, and therefore $\GG'$ is a complete system of closed geodesics with empty limit subcollection $\KK=\varnothing$. Assume now that $B$ contains at least a closed geodesic. By Theorem~\ref{t:scg2}, $B$ contains a finite collection of pairwise disjoint simple closed geodesics $\GG''$ such that:
\begin{itemize}

\item No complete geodesic is entirely contained in \[U:=B\setminus\bigcup_{\zeta\in\GG''}\zeta.\] 

\item If we denote by $\KK$ the subcollection of the waists in $\GG''$, and we set
\begin{align*}
 K:=\bigcup_{\gamma\in\KK} (\dot\gamma\cup-\dot\gamma),
\end{align*}
the trapped sets of $SU$ are given by
\begin{align*}
 \trap_+(SU)= \Ws(K)\setminus K,\qquad
 \trap_-(SU)= \Wu(K)\setminus K.
\end{align*}
\end{itemize}
Since every $\gamma\in\GG''\setminus\KK$ has conjugate points, Corollary~\ref{c:intersecting} implies that there exists $\ell>0$ such that, for every $z\in SM$ sufficiently close to $\dot\gamma\cup-\dot\gamma$, the orbit segment $\phi_{(0,\ell]}(z)$ is not contained in $SU$.  This proves that the union $\GG:=\GG'\cup\GG''$ is a complete system of  closed geodesics with limit subcollection~$\KK$.

Assume now that $M=S^2$ so that the collection $\GG'$ consists in just one simple closed geodesic $\gamma$ with conjugate points. Let $B_1$ and $B_2$ be the connected components of the complement $S^2\setminus\gamma$. Since $\gamma$ has conjugate points, each $B_i$ satisfies the convexity Assumption~\ref{a:convexity}. Corollary~\ref{c:intersecting} implies that there exists $\ell>0$ such that, for every $z\in SM$ sufficiently close to $\dot\gamma\cup-\dot\gamma$, the orbit segment $\phi_{(0,\ell]}(z)$ is not contained in $S(B_1\cup B_2)$.
If $S^2\setminus\gamma$ does not contain any simple closed geodesic, Theorem~\ref{t:scg1} implies $\trap_\pm(SB_1)\cup\trap_\pm(SB_2)=\varnothing$,  and therefore $\GG'$ is a complete system of closed geodesics with empty limit subcollection $\KK=\varnothing$. If  $S^2\setminus\gamma$ contains closed geodesics, we argue as in the previous paragraph in both $B_1$ and $B_2$, and find another collection $\GG''$ of pairwise disjoint simple closed geodesics in $B_1\cup B_2$ that, together with $\GG'$, form a complete system of closed geodesics $\GG:=\GG'\cup\GG''$ whose limit subcollection~$\KK$ is given by the waists in $\GG''$.
\end{proof}

We now have all the ingredients to prove Theorem~\ref{mt:Kupka_Smale}, which we restate for the reader's convenience.\vspace{7pt}

\noindent\textbf{Theorem \ref{mt:Kupka_Smale}.}
\emph{Let $(M,g)$ be a closed connected orientable Riemannian surface satisfying the following two conditions:}
\begin{itemize}

\item[(i)] \emph{All contractible simple closed geodesics without conjugate points are non-degenerate.}

\item[(ii)] \emph{Any pair of not necessarily distinct contractible  waists $\gamma_1,\gamma_2$ $($if it exists$)$ satisfies the transversality condition $\Wu(\dot\gamma_1)\pitchfork\Ws(\dot\gamma_2)$.}

\end{itemize}
\emph{Then the geodesic vector field of $(M,g)$ admits a Birkhoff section.}

\begin{proof}
Lemma~\ref{l:existence_complete_system} provides a complete system of closed geodesics. If its limit subcollection  is non-empty, by applying Lemma~\ref{l:reduction} a finite number of times, we end up with another complete system of closed geodesics $\GG$ with empty limit subcollection.
We denote by $\Upsilon$ the union of the Birkhoff annuli of the closed geodesics in $\GG$, i.e.
\begin{align*}
 \Upsilon:=\bigcup_{\gamma\in\GG} \big( A(\dot\gamma)\cup A(-\dot\gamma) \big).
\end{align*}
By Lemma~\ref{l:empty_limit_subcollection}, there exists $\ell>0$ such that, for each $z\in SM$, the orbit segment $\phi_{(0,\ell]}(z)$ intersects $\Upsilon$. We apply Fried surgery to resolve the self-intersection of $\Upsilon$ (if there is any), and end up with a surface of section $\Sigma\looparrowright SM$ such that, for each $z\in\Upsilon$, the orbit segment $\phi_{(-1,1)}(z)$ intersects $\Sigma$. Therefore, for each $z\in SM$, the orbit segment $\phi_{(0,\ell+2)}(z)$ intersects $\Sigma$, and we conclude that $\Sigma$ is a Birkhoff section.
\end{proof}

Finally, Theorem~\ref{mt:broken_book} is a consequence of the following statement.
\begin{Thm}
Let $(M,g)$ be a closed connected orientable surface. If $M$ has genus $G\geq1$, we consider the collection of waists $\GG'=\{\gamma_1,...,\gamma_{2G}\}$ given by Lemma~\ref{l:simple_closed_geodesics}, and the surface of section $\Sigma'\looparrowright SM$ provided by Lemma~\ref{l:Sigma}, which has genus one and $8G-4$ boundary components, all covering the closed geodesics in $\GG'$; if instead $M=S^2$, we set $\GG'=\{\gamma_0\}$, where $\gamma_0$ is any simple closed geodesic with conjugate points, and set $\Sigma':=A(\dot\gamma_0)\cup A(-\dot\gamma_0)$. 
Assume that the complement 
\[B:=M\setminus\bigcup_{\gamma\in\GG'}\gamma\] contains at least a closed geodesic and no degenerate simple closed geodesics without conjugate points. Then, there exists a finite collection $\GG''$ of simple closed geodesics that are pairwise disjoint and disjoint from the closed geodesics in $\GG'$ with the following properties. 
\begin{itemize}

\item[(i)] Every orbit of the geodesic flow intersects $\Sigma'\cup\Sigma''$, where 
\[\Sigma'':=\bigcup_{\gamma\in\GG''} \big(A(\dot\gamma)\cup A(-\dot\gamma)\big) \subset SM\setminus\Sigma'.\]

\item[(ii)] We denote by $\KK$ the subcollection of waists in $\GG'$, and 
\begin{align*}
 K:=\bigcup_{\gamma\in\KK} \big( \dot\gamma\cup-\dot\gamma \big).
\end{align*}
The trapped sets of $SM\setminus(\Sigma'\cup\Sigma'')$ are given by
\begin{align*}
\trap_+(SM\setminus(\Sigma'\cup\Sigma'')) &=\Ws(K)\setminus K,
\\
\trap_-(SM\setminus(\Sigma'\cup\Sigma'')) &=\Wu(K)\setminus K.
\end{align*}

\item[(iii)] There exists $\ell>0$ such that, for each $z\in SM$ sufficiently close to the boundary components $\partial\Sigma'\cup\partial\Sigma''\setminus K$, we have $\phi_{(0,\ell]}(z)\cap(\Sigma'\cup\Sigma'')\neq\varnothing$.

\end{itemize}
\end{Thm}

\begin{proof}
Lemma~\ref{l:existence_complete_system} provides a (possibly empty) finite collection $\GG''$ of contractible simple closed geodesics that are pairwise disjoint and disjoint from all the closed geodesics in $\GG'$, such that $\GG:=\GG'\cup\GG''$ is a complete system of closed geodesics whose limit subcollection $\KK$ is given by the waists in $\GG''$. We set
\begin{align*}
 \Upsilon':=\bigcup_{\gamma\in\GG'} \big( A(\dot\gamma)\cup A(-\dot\gamma) \big),
 \qquad
 \Sigma'':=\bigcup_{\gamma\in\GG''} \big( A(\dot\gamma)\cup A(-\dot\gamma) \big). 
\end{align*}
If $M=S^2$, we have set $\Sigma':=\Upsilon'$. If instead $M\neq S^2$, we require the surface of section $\Sigma'$ provided by Lemma~\ref{l:simple_closed_geodesics} to be contained in a sufficiently small neighborhood of $\Upsilon'$ so that $\Sigma'\cap\Sigma''=\varnothing$, and by Lemma~\ref{l:Sigma}(iv) we have
\begin{align*}
 \phi_{(-1,1)}(z)\cap\Sigma'\neq\varnothing,
 \ \
 \forall z\in\Upsilon',
 \qquad
 \phi_{(-1,1)}(z')\cap\Upsilon'\neq\varnothing,
 \ \
 \forall z'\in\Sigma'.
\end{align*}
We set
\begin{align*}
 K:=\bigcup_{\gamma\in\KK} (\dot\gamma\cup-\dot\gamma),
 \qquad
 R:=\!\!\!\!\!\bigcup_{\gamma\in\GG'\cup\GG''\setminus\KK}(\dot\gamma\cup-\dot\gamma)=\partial\Sigma'\cup\partial\Sigma''\setminus K.
\end{align*}
By properties~(ii) in Definition~\ref{d:complete_system}, 
every complete orbit $\phi_{(-\infty,\infty)}(z)$ intersects $\Sigma'\cup\Sigma''$, and we have
\begin{align*}
 \trap_+(SM\setminus(\Sigma'\cup\Sigma''))&=\trap_+(SM\setminus(\Upsilon'\cup\Sigma''))=\Ws(K)\setminus K,
 \\
 \trap_-(SM\setminus(\Sigma'\cup\Sigma''))&= \trap_-(SM\setminus(\Upsilon'\cup\Sigma''))=\Wu(K)\setminus K.
\end{align*}
Moreover, by properties~(iii) in Definition~\ref{d:complete_system}, there exist $\ell>0$ and an open neighborhood $N\subset SM$ of $R$ such that
\begin{align*}
\phi_{(0,\ell]}(z)\cap(\Upsilon'\cup\Sigma'')\neq\varnothing,
\qquad
\forall z\in N.
\end{align*}
If we take a sufficiently small open neighborhood $N'\subset N$ of $R$, we also have
\[
 \phi_{(0,\ell+1]}(z)\cap(\Sigma'\cup\Sigma'')\neq\varnothing,
\qquad
\forall z\in N'.
\qedhere
\]
\end{proof}

\bibliography{biblio}

\end{document}